\documentclass[11pt]{article}
\usepackage{amssymb, amsmath, amsthm, oldgerm,srcltx}
\usepackage{graphicx}
\usepackage{xcolor}
\newtheorem{thm}{Theorem}[section]
\newtheorem{lem}{Lemma}[section]
\newtheorem{cor}[lem]{Corollary}
\newtheorem{prop}[thm]{Proposition}
\newtheorem{rem}[thm]{Remark}
\numberwithin{equation}{section}

\newcommand{\abs}[1]{\left\vert#1\right\vert}

\newcommand{\eps}{\varepsilon}

\newcommand{\E}{\mathbf{E}\,}

\newcommand{\R}{\mathbf{R}}

\newcommand{\im}{\mathrm{Im}\;\!}
\newcommand{\Tr}{\mathrm{Tr}\;\!}
\newcommand{\q}{\quad}

\newenvironment{Proof of}{\removelastskip\par\medskip
\noindent{\em Proof of} \rm}{\penalty-20\null\hfill$\square$\par\medbreak}

\def\be{\begin{equation}}
\def\en{\end{equation}}
\def\bee{\begin{eqnarray*}}
\def\ene{\end{eqnarray*}}

\def\E{{\bf E}}

\def\R{{\mathbb R}}

\def\Tr{{\rm Tr}\,}
\def\Im{{\rm Im}\,}

\def\<{\left<}
\def\>{\right>}

\def\1{{\bf 1}}
\def\4{\kern1pt}

\begin{document}
\bibliographystyle{}

\vspace{1in}
 \date{}
\title{\bf Rate of  Convergence of the Empirical Spectral Distribution Function to the Semi-Circular Law}

\author{{\bf F. G\"otze}\\{\small Faculty of Mathematics}
\\{\small University of Bielefeld}\\{\small Germany}
\and {\bf A. Tikhomirov}$^{1}$\\{\small Department of Mathematics
}\\{\small Komi Research Center of Ural Division of RAS,}\\{\small Syktyvkar State University}
\\{\small  Russia}
}
\maketitle
 \footnote{$^1$Research supported   by SFB 701 ``Spectral Structures and Topological Methods in Mathematics'' University of Bielefeld.
Research supported   by grants  RFBR N~14-01-00500  and by Program of Fundamental Research Ural Division of RAS, Project ¹ 12-P-1-1013}

\date{}

\maketitle
\begin{abstract}
Let $\mathbf X=(X_{jk})_{j,k=1}^n$ denote a  Hermitian random matrix with
entries $X_{jk}$, which are independent for $1\le j\le  k\le n$. We consider
the rate of convergence of the empirical  spectral distribution function of
the matrix $\mathbf W=\frac1{\sqrt n}\mathbf X$ to the semi-circular law assuming that $\E X_{jk}=0$,
$\E X_{jk}^2=1$ and that 
\begin{equation}
\sup_{n\ge1}\sup_{1\le j,k\le n}\E|X_{jk}|^4=: \mu_4<\infty
\text{  and  }
\sup_{1\le j,k\le n}|X_{jk}|\le Dn^{\frac14}.
\end{equation}
By means of a recursion argument it is shown that the
 Kolmogorov distance between the empirical spectral
distribution of the Wigner matrix $\mathbf W$
and the  semi--circular law is of order $O(n^{-1}\log^{5}n)$ with high probability.
\end{abstract}
\maketitle


\section{Introduction}
\setcounter{equation}{0}

 The present paper is a continuation of the paper \cite{GT:2014}, where we   proved   under the assumptions of Theorem \ref{main}
below a non improvable bound for the Kolmogorov distance between the expected spectral distribution function of Wigner matrices  and the semicircular 
distribution function. In this paper we estimate the $L_p$-norm of the Kolmogorov distance between the empirical spectral distribution function of Wigner matrices  
and the semicircular distribution function, for $1\le p\le C\log n$.

Consider a family $\mathbf X = \{X_{jk}\}$, $1 \leq j \leq k \leq n$,
of independent real random variables defined on some probability space
$(\Omega,{\textfrak M},\Pr)$, for any $n\ge 1$. Assume that $X_{jk} = X_{kj}$, for
$1 \leq k < j \leq n$, and introduce the symmetric matrices
\begin{displaymath}\notag
 \mathbf W = \ \frac{1}{\sqrt{n}} \left(\begin{array}{cccc}
 X_{11} &  X_{12} & \cdots &  X_{1n} \\
 X_{21} & X_{22} & \cdots &  X_{2n} \\
\vdots & \vdots & \ddots & \vdots \\
 X_{n1} &  X_{n2} & \cdots &  X_{nn} \\
\end{array}
\right).
\end{displaymath}

The matrix $\mathbf W$ has a random spectrum $\{\lambda_1,\dots,\lambda_n\}$ and an
associated spectral empirical distribution function
$\mathcal F_{n}(x) = \frac{1}{n}\ {\rm card}\,\{j \leq n: \lambda_j \leq
x\}, \quad x \in \R$.
Averaging over the random values $X_{ij}(\omega)$, define the expected
(non-random) empirical distribution functions
$ F_{n}(x) = \E\,\mathcal F_{n}(x)$.
Let $G(x)$ denote the semi-circular distribution function with density
$g(x)=G'(x)=\frac1{2\pi}\sqrt{4-x^2}\mathbb I_{[-2,2]}(x)$, where $\mathbb I_{[a,b]}(x)$
denotes the indicator--function of the interval $[a,b]$. 
 The rate of convergence to the semi-circular law has been studied by several authors.  For a  detailed discussion of  previous results see  \cite{GT:2014}.
 
 We shall estimate  the  Kolmogorov distance between $\mathcal F_n(x)$ and
 the distribution function
 $G(x)$,
 $\Delta_n^*:=\sup_x|\mathcal F_n(x)-G(x)|$. 

The main result of this paper is the following
\begin{thm}\label{main} Let $\E X_{jk}=0$, $\E X_{jk}^2=1$.  Assume that
there exists a constant $\mu_4>0$ such that   
\begin{equation}\label{moment}
 \sup_{n\ge1}\,\sup_{1\le j,k\le n}\E|X_{jk}|^4=: \mu_4<\infty.
\end{equation}
 Furthermore, assume that there exists a constant $D$ such that for all $n$
\begin{equation}\label{trun}
 \sup_{1\le j,k\le n}|X_{jk}|\le Dn^{\frac14}.
\end{equation}

Then,   there exist  positive  constants $C=C(D,\mu_4)$ and $C'=C'(D,\mu_4)$ depending on
 $D$ and $\mu_4$ only
such that, for $p\le C'\log n$
\begin{equation} \label{kolmog}
\E^{\frac1p}{\Delta_n^*}^p\le Cn^{-1}\log^4 n.
\end{equation}

\end{thm}
\begin{cor}\label{cormain}Let $\E X_{jk}=0$, $\E X_{jk}^2=1$.  Assume that
\begin{equation}\label{moment1}
 \sup_{n\ge1}\,\sup_{1\le j\le k\le n}\E|X_{jk}|^8=: \mu_8<\infty.
\end{equation}
 Then,    there exists  positive  constants $C=C(\mu_8)$  and $C'=C'(\mu_8)$ depending on  
 $\mu_8$   only
such that, for $p\le C'\log n$
\begin{equation} \label{kolmog1}
\E^{\frac1p}{\Delta_n^*}^p\le Cn^{-1}\log^4 n.
\end{equation}
\end{cor}
\begin{cor} Assume that conditions \eqref{moment} and \eqref{trun} or \eqref{moment1} hold.
Then there exist  positive constants $C,c$ depending on $D, \mu_4$ or $\mu_8$ only such that  
 \begin{equation}\notag
  \Pr\{\Delta_n^*\ge Cn^{-1}\log^5 n\}\le n^{-c\log\log n}.
 \end{equation}

\end{cor}
\begin{proof} 
 The result follows immediately from Theorem \ref{main} or Corollary \ref{cormain} and Chebyshev's inequality.
\end{proof}
\begin{cor}Inequality \eqref{kolmog1} implies that
 \begin{align} 
\Pr\Bigl\{\, \exists\,\,j&\in[c_1\log^5 n,n-c_1\log^5 n]:\notag\\& \,|\lambda_j-\gamma_{nj}|\ge C\log^5 n
\Big[\min\{j,n-j+1\}\Big]^{-\frac13}n^{-\frac23} \Bigr\}  \le n^{-c\log\log n}. \notag
\end{align}
\end{cor}
\begin{proof}For a proof of this Corollary see  Subsection 9.1, Appendix  in \cite{GT:2013}. 
This shows  the localization rigidity  of eigenvalues except 
for a neighborhood of the  edges given by $k\le C\log^5n$ or $k\ge n-C\log^5n$.
\end{proof}

 We denote the Stieltjes transform of $\mathcal F_n(x)$
by $m_n(z)$ and the Stieltjes transform of the semi-circular law by $s(z)$.  Let
$\mathbf R=\mathbf R(z)$ be the resolvent matrix of $\mathbf W$ given by
$\mathbf R=(\mathbf W-z\mathbf I_n)^{-1}$,
for all $z=u+iv$ with $v\ne 0$. Here and in what follows $\mathbf I_n$ denotes the identity matrix of dimension $n$.
Sometimes we shall omit the sub index in the notation of an identity  matrix.
It is well-known that the Stieltjes transform of the semi-circular distribution
satisfies the equation
\begin{equation}\label{stsemi}
s^2(z)+zs(z)+1=0
\end{equation}
(see, for example, equality (4.20) in \cite{GT:03}). 
Let 
\begin{equation}\label{v0}
v_0:= A_0n^{-1}\log^4n
\end{equation}
 and $\gamma(z):=|2-|u||$, for $z=u+iv$.
Introduce the region 
$\mathbb G=\mathbf G(A_0, n,\varepsilon)\subset\mathbb C_+$ 
\begin{equation}\notag
 \mathbb G:=\{z=u+iv\in\mathbb C_+: -2+\varepsilon\le u\le 2-\varepsilon,\, v\ge v_0/\sqrt{\gamma(z)}\}.
\end{equation}
Let  $a>0$  be positive number such that
\begin{equation} \label{constant}
 \frac1{\pi}\int_{|u|\le a}\frac1{u^2+1}du=\frac34.
\end{equation}
We prove the following result.
\begin{thm}\label{stiltjesmain}Let  $\frac12>\varepsilon>0$ be positive numbers such that
\begin{equation}\label{avcond}
 \varepsilon^{\frac32}:=2v_0a.
\end{equation}
Assuming the conditions of Theorem \ref{main}, for any $A_1>0$ there exist  positive constants $C=C(D,\mu_4, A_1)$ and $A_0=A_0(\mu_4,D,A_1)$ depending on $D$, $A_1$ 
and $\mu_4$ only, such that, for $z\in\mathbb G$ and for $1\le p\le A_1(nv)^{\frac14}$
\begin{align}\notag
 \E|m_n(z)-s(z)|^p\le (Cp)^p(nv)^{-p}.
\end{align}

\end{thm}

\begin{cor}\label{stieltjesmaincor}Let  $\frac12>\varepsilon>0$ be positive numbers such that the condition \eqref{avcond} holds.
Let $\E X_{jk}=0$, $\E X_{jk}^2=1$.  Assume that
there exists a constant $\mu_8>0$ such that   for any $1\le j\le k\le n$ 
\begin{equation}\notag
 \sup_{j,k}\E|X_{jk}|^8=: \mu_8<\infty.
\end{equation}
 Then for any $A_1>0$ there exist  positive constants $C=C(\mu_8, A_1)$ and $A_0=A_0(\mu_8,A_1)$  depending on $\mu_8$ and $A_1$ only, such that, for $z\in\mathbb G$
 and $1\le p\le A_1(nv)^{\frac14}$,
\begin{align}\notag
 \E|m_n(z)-s(z)|^p\le (Cp)^p(nv)^{-p}.
\end{align}
\end{cor}
Similar results were obtained recently in \cite{SchleinMaltseva:2013}, Theorems 1,2, assuming sub-Gaussian tails for the distribution of the matrix entries.

\subsection{Sketch of the Proof}
1. As in \cite{GT:2014} we start with an estimate of the Kolmogorov-distance to the Wigner distribution via
an integral over the difference of the corresponding Stieltjes transforms along a contour 
in the upper half-plane using a  smoothing inequality \eqref{smoothing11}. This inequality is adapted   to the $L_p$-norm of the corresponding Kolmogorov distance.
The resulting bound \eqref{smoothing11} involves an integral over a segment, say 
$V=4$, at a fixed distance from the real axis and a segment  $u+i A_0 n^{-1}(2-\abs{u})^{-\frac 12}, \; u\le x$
at a distance of order $n^{-1}\log^4n$ but avoiding to come close to the endpoints $\pm 2$ of the support.
These segments are part of the boundary of an $n$-dependent region $\mathbb G$ where
bounds of Stieltjes transforms are needed. Since the Stieltjes-transform
and the diagonal elements $R_{jj}(z)$ of the resolvent  of the Wigner-matrix $\mathbf W$ are uniformly bounded 
on the segment with $\Im z=V$  by $1/V$ (see Section \ref{firsttypeint})  proving a bound of order 
$O(n^{-1}\log n)$  for the latter segment near the x-axis is the essential problem.

{\noindent \bf 2.} In order to investigate this crucial part of the error 
we start  with the 2nd resolvent  or self-consistency equation for the  Stieltjes transform resp.
the quantities $R_{jj}(z)$ of $\mathbf W$ (see \eqref{repr001} below) based on the difference of the resolvent 
of $\mathbf W^{(j)}$ ($j$th row and column removed) and $\mathbf W$.  
The  necessary bounds of
$\E|R_{jj}|^p$ for large $p=O(\log n)$ were proved in \cite{GT:2014}.

{\noindent \bf 3.} In Section \ref{expect} we prove a bound  for the error $\E|\Lambda_n|^p:= \E|m_n(z)-s(z)|^p$ of the
form 
$ C^pp^p(n v)^{-p}$ for $p\le C'(nv)^{\frac14}$
which suffices to prove the rate $O(n^{-1}\log^4 n)$
in Theorem  \ref{main}. Here we use a series of martingale-type
decompositions to evaluate  $\E|\Lambda_n|^p$.

{\noindent \bf 4.}
The necessary auxiliary bounds for all these steps
are collected in the Appendix.

\section{Bounds for the  Kolmogorov Distance Between Distribution Functions  via  Stieltjes Transforms}\label{smoothviastil}
To bound the error $\Delta_n^*$ we shall use an approach developed in previous work of the authors, see \cite{GT:03}.\\
We modify the bound of the  Kolmogorov distance between an arbitrary distribution function and the semi-circular distribution function
 via their Stieltjes transforms obtained in \cite{GT:03} Lemma 2.1. For $x\in[-2,2]$ define $\gamma(x):=2-|x|$.
 Given $\frac12>\varepsilon>0$ introduce the interval $\mathbb J_{\varepsilon}=\{x\in[-2,2]:\, \gamma(x)\ge\varepsilon\}$ and
$\mathbb J'_{\varepsilon}=\mathbb J_{\varepsilon/2}$.
For a distribution function $F$ denote by $S_F(z)$ its Stieltjes transform,
$$
S_F(z)=\int_{-\infty}^{\infty}\frac1{x-z}dF(x).
$$
\begin{prop}\label{smoothing1}
 Under the conditions of Proposition \ref{smoothing1} the following inequality holds
\begin{align}
 \Delta(F,G)&\le 2\int_{-\infty}^{\infty}|S_F(u+iV)-S_G(u+iV)|du+C_1v_0+C_2\varepsilon^{\frac32}\notag\\&
  + 2 \sup_{x\in\mathbb J'_{\varepsilon}}\Big|\int_{v'}^V(S_F(x+iu)-S_G(x+iu))du\Big|,\notag
\end{align}
where $v'=\frac {v_0}{\sqrt{\gamma}}$ with $\gamma=2-|x|$ and $C_1,C_2 >0$ denote absolute constants.

\end{prop}
\begin{rem}\label{rem2.2}
 For any $x\in\mathbb J_{\varepsilon}$ we have
$\gamma=\gamma(x)\ge\varepsilon$
and according to condition \eqref{avcond},
$\frac{av}{\sqrt\gamma}\le \frac{\varepsilon}2$.
\end{rem}

For a  proof of this Proposition see  \cite{GT:2013}, Proposition 2.1.

\begin{cor}\label{smoothing2}
 Under the conditions of Proposition \ref{smoothing1} the following inequality holds
\begin{align}\label{smoothing11}
 \E^{\frac1p}\Delta^p(\mathcal F_n,G)&\le 2\int_{-\infty}^{\infty}\E^{\frac1p}|m_n(u+iV)-s(u+iV)|^pdu+C_1v_0+C_2\varepsilon^{\frac32}\notag\\&
  + 2 \E^{\frac1p}\sup_{x\in\mathbb J'_{\varepsilon}}\Big|\int_{v'}^V(m_n(u+iV)-s(u+iV))du\Big|^p,
\end{align}
where $v'=\frac {v_0}{\sqrt{\gamma}}$ with $\gamma=2-|x|$ and $C_1,C_2 >0$ denote absolute constants.
\end{cor}
\begin{proof}
 To prove this Corollary we observe that by H\"older's inequality
 \begin{align}\label{integral}
  \E\Big[\int_{-\infty}^{\infty}&|m_n(u+iV)-s(u+iV)|du\Big]^p\notag\\&\le \E\prod_{l=1}^p\int_{-\infty}^{\infty}|m_n(u_l+iV)-s(u_l+iV)|du_l\notag\\&=
  \int_{-\infty}^{\infty}\cdots\int_{-\infty}^{\infty}\E\prod_{l=1}^p\int_{-\infty}^{\infty}|m_n(u_l+iV)-s(u_l+iV)|du_1\cdots du_p\notag\\&\le
\int_{-\infty}^{\infty}\cdots\int_{-\infty}^{\infty}\prod_{l=1}^p\E^{\frac1p}|m_n(u_l+iV)-s(u_l+iV)|^pdu_1\cdots du_p\notag\\&\le
\prod_{l=1}^p\int_{-\infty}^{\infty}\E^{\frac1p}|m_n(u_l+iV)-s(u_l+iV)|^pdu_l\notag\\&=\Big[\int_{-\infty}^{\infty}\E^{\frac1p}|m_n(u+iV)-s(u+iV)|^pdu\Big]^p.
  \end{align}
Proposition \ref{smoothing1} and inequality \eqref{integral} together conclude the proof of Corollary \ref{smoothing2}.
\end{proof}

\section{The proof of Theorem \ref{main}}
 We shall apply  the Corollary \ref{smoothing2} to prove the Theorem \ref{main}.
 We choose $V=4$ and $v_0$ as defined in \eqref{v0} and introduce the quantity 
 $\varepsilon=(2av_0)^{\frac23}$ .
 We shall denote in what follows by $C$  a generic constant depending on $\mu_4$ and $D$ only.
\subsection{Estimation of the First Integral in \eqref{smoothing2}  for $V=4$}\label{firsttypeint}
Denote by $\mathbb T=\{1,\ldots,n\}$. In the following we shall systematically
  use for any  $n\times n $ matrix $\mathbf X$
 together with its resolvent $\mathbf R$, its Stieltjes transform $m_n$ etc. the corresponding quantities $\mathbf X^{(\mathbb A)}$, $\mathbf R^{(\mathbb A)}$ 
 and $m_n^{(\mathbb A)}$
 for the corresponding  sub matrix  with entries $X_{jk}, j, k \not \in \mathbb A$, $\mathbb A \subset \mathbb T=\{1,\ldots,n\}$.
Observe that
\begin{equation}
 m_n^{(\mathbb A)}(z)=\frac1n\sum_{j\in\mathbb T_{\mathbb A}}\frac1{\lambda^{(\mathbb A)}-z}.
\end{equation}
By $\mathfrak M^{(\mathbb J)}$ we denote the $\sigma$-algebra generated by $X_{lk}$ with $l,k\in\mathbb T_{\mathbb J}$.
If $\mathbb A=\emptyset$ we shall omit the set $\mathbb A$ as exponent index. 

In this Section we shall consider $z=u+iV$ with $V=4$. We shall use the representation
\be\notag
R_{jj}=\frac1{-z+\frac1{\sqrt
n}X_{jj}-\frac1n{{\sum_{k,l\in\mathbb T_j}}}X_{jk}X_{jl}R^{(j)}_{kl}},
\en
(see,
for example, equality (4.6) in \cite{GT:03}). We may rewrite it as
follows
\begin{equation}\label{repr001*}
 R_{jj}=-\frac1{z+m_n(z)}+\frac1{z+m_n(z)}\varepsilon_jR_{jj},
\end{equation}

where
$\varepsilon_j:=\varepsilon_{j1}+\varepsilon_{j2}+\varepsilon_{j3}+
\varepsilon_{j4}$   with
\begin{align}
\varepsilon_{j1}&:=\frac1{\sqrt n}X_{jj},\quad
\varepsilon_{j2}:=-\frac1n{\sum_{k\ne l
\in\mathbb T_j}}X_{jk}X_{jl}R^{(j)}_{kl},\quad
\varepsilon_{j3}:=-\frac1n{\sum_{k\in\mathbb T_j}}(X_{jk}^2-1)R^{(j)}_{kk},\notag\\
\varepsilon_{j4}&:=\frac1n(\Tr \mathbf R-\Tr\mathbf R^{(j)}).\quad
\end{align}
Let
\begin{equation}\notag
\Lambda_n:=\Lambda_n(z):=m_n(z)-s(z)=\frac1n\Tr\mathbf R-s(z).    
\end{equation}
As follows from \eqref{stsemi}, for the semi-circular law we have
\begin{equation}\label{scmain0}
 s(z)=-\frac1{z+s(z)}\text{ and }|s(z)|\le 1.
\end{equation}
See, for instance \cite{Bai:93}, p. 632, relations (3.2), (3.3).
Summing equality \eqref{repr001*} in $j=1,\ldots,n$ and solving with respect $\Lambda_n$, we get 
\begin{equation}\label{lambda'}
\Lambda_n=  m_n(z)-s(z)=\frac{T_n}{z+m_n(z)+s(z)},
\end{equation}
where
\begin{equation}\notag
 T_n=\frac1n\sum_{j=1}^n\varepsilon_jR_{jj}.
\end{equation}
Note that for $V=4$
\begin{equation}\label{3.5}
 |m_n(z)|\le \frac14\le \frac12|z+s(z)|, \quad |s(z)-m_n(z)|\le \frac12 \text{ a.s.}
\end{equation}
This implies
\begin{equation}\label{inequal10}
 |z+m_n(z)+s(z)|\ge\frac12|z+s(z)|,\quad |z+m_n(z)|\ge \frac12|s(z)+z|.
\end{equation}
The last inequality and inequality \eqref{lambda'} imply as well that, for $V=4$,
\begin{equation}\label{mnz}
|m_n(z)|\le |s(z)|(1+2|T_n(z)|).
\end{equation}

Let
\begin{equation}\notag
 \varphi(z)=\overline z|z|^{p-2}.
\end{equation}
Using  equality \eqref{lambda'}, we may write, for $p\ge2$,
\begin{align}\notag
 \E|\Lambda_n|^p=\sum_{\nu=1}^4\E\frac{T_{n\nu}}{z+m_n(z)+s(z)}\varphi(\Lambda_n),
\end{align}
where
\begin{equation}\notag
 T_{n\nu}=\frac1n\sum_{j=1}^n\varepsilon_{j\nu}R_{jj}.
\end{equation}

We consider first the term with $\nu=4$.
Using the relation
\begin{equation}\label{3.9}
 \frac1n\sum_{j=1}^n\varepsilon_{j4}R_{jj}=-\frac1n\Tr\mathbf R^2,
\end{equation}
(see Lemma 5.5 in \cite{GT:2013}) and \eqref{inequal10}), we get, using Lemma \ref{resol00}, inequality \eqref{res1} in the Appendix, and inequalities 
\eqref{inequal10} and  \eqref{mnz}
\begin{align}\label{finish1}
 |\frac{T_{n4}}{z+m_n(z)+s(z)}|&\le \frac1n|\sum_{j=1}^n\frac{\varepsilon_{j4}R_{jj}}{z+m_n(z)+s(z)}|\le \frac {C|s(z)|}n|m_n(z)|\notag\\&\le \frac Cn|s(z)|^2(1+|T_n|).
\end{align}
Therefore, by H\"older's inequality,
\begin{align}
 \Big|\E\frac{T_{n4}\varphi(\Lambda_n)}{z+m_n(z)+s(z)}\Big|\le \frac{C|s(z)|^{2}}n(1+\E^{\frac1p}|T_n|^p)
 \E^{\frac{p-1}p}|\varphi(\Lambda_n)|^{\frac{p}{p-1}}.\notag
\end{align}
It is straightforward to check that, by the Cauchy -- Schwartz inequality and  $|R_{jj}|\le V^{-1}$,
\begin{equation}\notag
 \E|T_n|^p\le \E\Big(\frac1n\sum_{j=1}^n|\varepsilon_j|^2\Big)^{\frac p2}\Big(\frac1n\sum_{j=1}^n|R_{jj}|^2\Big)^{\frac p2}
 \le \frac1{2^p}\E\Big(\frac1n\sum_{j=1}^n|\varepsilon_j|^2\Big)^{\frac p2}.
\end{equation}
Applying Corollary \ref{eps1} and Lemmas \ref{eps2} and \ref{eps3} in the Appendix, we get that there exists an absolute  constant $C''>0$ 
such that for $p\le C'\log n$,
\begin{align}\notag
 \E^{\frac1p}|T_n|^p\le Cpn^{-\frac 12}\le C''CC'\le C.
\end{align}
Therefore,
\begin{equation}\notag
 \Big|\E\frac{T_{n4}\varphi(\Lambda_n)}{z+m_n(z)+s(z)}\Big|\le \frac{C|s(z)|^{2}}n
 \E^{\frac{p-1}p}|\varphi(\Lambda_n)|^{\frac{p}{p-1}}.
\end{equation}

Furthermore, we represent, for $\nu=1,2,3$
\begin{align}
 \frac1n\E\frac{\sum_{j=1}^n\varepsilon_{j\nu}
 R_{jj}\varphi(\Lambda_n)}{z+m_n(z)+s(z)}=H_1+H_2,\notag
\end{align}
where
\begin{align}\label{finish01}
 H_1&:=\frac1n\E\frac{\sum_{j=1}^n\varepsilon_{j\nu}
s(z)\varphi(\Lambda_n)}{z+m_n(z)+s(z)},\notag\\
H_2&:=\frac1n\E\frac{\sum_{j=1}^n\varepsilon_{j\nu}(R_{jj}-
s(z))\varphi(\Lambda_n)}{z+m_n(z)+s(z)},
\end{align}
First we bound $H_2$. Applying first the Cauchy -- Schwartz inequality followed by H\"older's inequality, we get
\begin{align}\label{eps0}
 |H_2|\le {C|s(z)|}\E^{\frac1{2p}}\Big(\frac1n\sum_{j=1}^n|\varepsilon_{j\nu}|^{2}\Big)^p\E^{\frac1{2p}}\Big(\frac1n\sum_{j=1}^n|R_{jj}-s(z)|^2\Big)^{p}
\E^{\frac{p-1}p}|\varphi(\Lambda_n)|^{\frac p{p-1}}.
\end{align}

Using the representation \eqref{repr001*}, we may write
\begin{align}\label{lambda''}
 R_{jj}=s(z)-s(z)\varepsilon_jR_{jj}-s(z)\Lambda_nR_{jj}.
\end{align}
Applying the representations \eqref{lambda''} and \eqref{lambda'}, we obtain
\begin{align}\label{eps}
 \frac1n\sum_{j=1}^n|R_{jj}(z)-s(z)|^{2}\le C|s(z)|^{2}\Big(\frac1n\sum_{l=1}^n|\varepsilon_l|^{2}\Big).
\end{align}
Combining inequalities \eqref{eps0} and \eqref{eps}, we get
\begin{align}\notag
 |H_2|\le C|s(z)|^2\E^{\frac1p}\Big(\frac1n\sum_{j=1}^n|\varepsilon_{j}|^2\Big)^{p}\E^{\frac{p-1}p}|\varphi(\Lambda_n)|^{\frac p{p-1}}.
\end{align}
Applying now Corollary \ref{eps1} and Lemmas \ref{eps2}, \ref{eps3} in the Appendix, we get
\begin{align}\label{finish2}
 |H_2|\le \frac{C|s(z)|^{2}}{n}\E^{\frac{p-1}p}|\varphi(\Lambda_n)|^{\frac p{p-1}}.
\end{align}
We continue with $H_1$ 
and  represent it in the form
\begin{align}\label{finish001}
 H_1:=H_{11}+H_{12}+H_{13},
\end{align}
where
\begin{align}
H_{11}&:=\frac1n\E\frac{\sum_{j=1}^n\varepsilon_{j\nu}
s(z)\varphi({\widetilde{\Lambda}_n^{(j)}})}{z+m_n^{(j)}(z)+s(z)},\notag\\ 
H_{12}&:=\frac1n\E\frac{\sum_{j=1}^n\varepsilon_{j\nu}
s(z)(\varphi(\Lambda_n)-\varphi(\widetilde{\Lambda}_n^{(j)}))}{z+m_n^{(j)}(z)+s(z)},\notag\\
H_{13}&:=-\frac1n\E\frac{\sum_{j=1}^n\varepsilon_{j\nu}
s(z)\varepsilon_{j4}\varphi(\Lambda_n)}{(z+m_n(z)+s(z))(z+m_n^{(j)}(z)+s(z))},\notag
\end{align}
where
\begin{equation}\notag
 \widetilde{\Lambda}_n^{(j)}=\Lambda_n^{(j)}+\frac{s(z)}n+\frac1{n^2}\Tr{\mathbf R^{(j)}}^2s(z).
\end{equation}

It is straightforward to check that, by conditional independence of $\varepsilon_{j\nu}$ and $\widetilde{\Lambda}_n^{(j)}$,
\begin{equation}\label{finish0001}
 H_{11}=0.
\end{equation}
Using Lemma \ref{lem2} in the Appendix and applying the Cauchy -- Schwartz inequality, H\"older's inequality and inequality \eqref{inequal10}, we get
\begin{align}\label{finish00001}
 |H_{13}|&\le \frac{C|s(z)|^{2}}{n}\E\Big(\frac1n\sum_{j=1}^n|\varepsilon_{j\nu}|^2\Big)^p\E^{\frac{p-1}p}|\varphi(\Lambda_n)|^{\frac p{p-1}}\le
 \frac{C|s(z)|^{2}}{n}\E^{\frac{p-1}p}|\varphi(\Lambda_n)|^{\frac p{p-1}}.
\end{align}

Let
\begin{equation} \label{important}
 \eta_{j0}=\frac1n\sum_{l\in\mathbb T_j}[(\mathbf R^{(j))^2]}_{ll}, \quad\eta_{j1}=\frac1p\sum_{l\in\mathbb T_j}(X_{jl}^2-1)[(\mathbf R^{(j))^2]}_{ll},\quad
 \eta_{j2}=\frac1n\sum_{k\ne l\in\mathbb T_j}X_{jk}X_{jl}[(\mathbf R^{(j)})^2]_{kl}.
\end{equation}
Note that
\begin{equation}\label{important1}
 |\eta_{j0}|=\frac1n|\Tr {\mathbf R^{(j)}}^2|\le v^{-1}\im m_n^{(j)}(z).
\end{equation}
We use  that (see Lemma 5.5 in \cite{GT:2013})
\begin{equation}\label{important2}
 \varepsilon_{j4}=\frac1n(1+\eta_{j0}+\eta_{j1}+\eta_{j2})R_{jj}.
\end{equation}
Note that 
\begin{equation}\notag
 \delta_{nj}:=\Lambda_n-\widetilde{\Lambda}_n^{(j)}=\varepsilon_{j4}-\frac {s(z)}n-\eta_{j0}s(z)=\frac1n(R_{jj}-s(z))(1+\eta_{j0})+\frac1n(\eta_{j1}+\eta_{j2})R_{jj}.
\end{equation}
This yields
\begin{align}\label{gg1}
 |\delta_{nj}|\le \frac1n(1+|\eta_{j0}|)|R_{jj}-s(z)|+\frac1n|\eta_{j1}+\eta_{j2}|)(|s(z)|+|R_{jj}-s(z)|).
\end{align}

By Taylor's formula $\varphi(x)-\varphi(y)=(x-y)\E_{\tau}\varphi'(x-\tau(x-y))$, we may write
\begin{align}\notag
 H_{12}=\frac{s(z)}n\E\frac{\sum_{j=1}^n\varepsilon_{j\nu}\delta_{nj}\varphi'(\Lambda_n-\tau\delta_{nj})
}{z+m_n^{(j)}(z)+s(z)},
\end{align}
where $\tau$ denotes a uniformly distributed random variable on the unit interval which is independent of all other random variables.
Note that, according to Lemma \ref{exp*} in the Appendix with $\zeta=\delta_{nj}$,
\begin{equation}\label{lala}
 |\varphi'(\Lambda_n-\tau\delta_{nj})|\le p|\Lambda_n-\tau\delta_{nj}|^{p-2}\le Cp|\Lambda_n|^{p-2}+p^{p-1}|\tau\delta_{nj}|^{p-2}.
\end{equation}
Therefore, applying inequality \eqref{inequal10}, the Cauchy -- Schwartz inequality, H\"older's inequality and finally inequality \eqref{lala}, we get
\begin{align}\notag
 |H_{12}|&\le C{p|s(z)|^2}\E^{\frac1{p}}(\frac1n\sum_{j=1}^n|\varepsilon_{j\nu}|^2)^{\frac p2}\E^{\frac1{p}}\Big(\frac1n\sum_{j=1}^n|\delta_{nj}|^2\Big)^{\frac p2}
 \E^{\frac{p-2}p}|\Lambda_n|^p\notag\\&\qquad\qquad\qquad\qquad+p^{p-1}|s(z)|^2\frac1n\sum_{j=1}^n\E|\varepsilon_{j\nu}||\delta_{nj}|^{p-1}.
\end{align}
Denote by $\eta_j:=\eta_{j1}+\eta_{j2}$.
Applying inequality \eqref{gg1} and that $|\eta_{j0}|\le V^{-2}$ , we get
\begin{align}\notag
 |H_{12}|&\le K_1+K_2+K_3+K_4,
 \end{align}
 where
 \begin{align}
 K_1&:=\frac{Cp|s(z)|^2}n\E^{\frac1{p}}(\frac1n\sum_{j=1}^n|\varepsilon_{j\nu}|^2)^{\frac p2}\E^{\frac1{p}}(\frac1n\sum_{j=1}^n|R_{jj}-s(z)|^2)^{\frac p2}
 \E^{\frac{p-2}p}|\Lambda_n|^p,\notag\\
 K_2&:=\frac{Cp|s(z)|^3}n\E^{\frac1{p}}(\frac1n\sum_{j=1}^n|\varepsilon_{j\nu}|^2)^{\frac p2}\E^{\frac1{p}}
 (\frac1n\sum_{j=1}^n|\eta_{j}|^2)^{\frac p2}\E^{\frac{p-2}p}|\Lambda_n|^p,\notag\\
 K_3&:=\frac{Cp|s(z)|^3}n\E^{\frac1{p}}(\frac1n\sum_{j=1}^n|\varepsilon_{j\nu}|^2)^{\frac p2}\E^{\frac1{2p}}
 (\frac1n\sum_{j=1}^n|\eta_{j}|^4)^{\frac p2}\notag\\&\qquad\qquad\qquad\qquad\times
 \E^{\frac1{2p}}(\frac1n\sum_{j=1}^n|R_{jj}-s(z)|^4)^{\frac p2}\E^{\frac{p-2}p}|\Lambda_n|^p,\notag\\
K_4&:=\frac{p^{p-1}|s(z)|^2}n\sum_{j=1}^n\E|\varepsilon_{j\nu}||\delta_{nj}|^{p-1}.\notag
\end{align}
Using equality \eqref{lambda''} and $|\Lambda_n|\le \frac12$ a. s., we get, for $z=u+iV$,
\begin{align}\label{lambdanew}
 |R_{jj}(z)-s(z)|&\le C|s(z)|^2(|\varepsilon_j|+|\varepsilon_j|^2+|\Lambda_n|+|\Lambda_n|^2)\notag\\&\le C|s(z)|^2(|\varepsilon_j|+|\varepsilon_j|^2+|\Lambda_n|).
\end{align}
Therefore,
\begin{align}
 \E^{\frac1{p}}\Big(\frac1n\sum_{j=1}^n|R_{jj}-s(z)|^2\Big)^{\frac p2}&\le C|s(z)|^4\Big(\E^{\frac1p}\Big(\frac1n\sum_{j=1}^n|\varepsilon_{j}|^2\Big)^{\frac p2}\notag\\&+
 \E^{\frac1p}\Big(\frac1n\sum_{j=1}^n|\varepsilon_{j}|^4\Big)^{\frac p2}+\E^{\frac1p}|\Lambda_n|^p\Big).
\end{align}
Applying the last inequality, we get
\begin{align}
 |K_1|&\le \frac{Cp|s(z)|^4}n\E^{\frac1{p}}(\frac1n\sum_{j=1}^n|\varepsilon_{j\nu}|^2)^{\frac p2}\E^{\frac1{p}}(\frac1n\sum_{j=1}^n|\varepsilon_{j}|^2)^{\frac p2}\E^{\frac{p-2}p}|\Lambda_n|^p\notag\\&+
 \frac{Cp|s(z)|^4}n\E^{\frac1{p}}(\frac1n\sum_{j=1}^n|\varepsilon_{j\nu}|^2)^{\frac p2}
 \E^{\frac1{p}}(\frac1n\sum_{j=1}^n|\varepsilon_{j}|^4)^{\frac p2}\E^{\frac{p-2}p}|\Lambda_n|^p\notag\\&+
 \frac{Cp|s(z)|^4}n\E^{\frac1{p}}(\frac1n\sum_{j=1}^n|\varepsilon_{j\nu}|^2)^{\frac p2}\E^{\frac{p-1}p}|\Lambda_n|^p.\notag
\end{align}

Using Corollary \ref{eps1} and Lemmas \ref{eps2}, \ref{eps3}, in the Appendix, we get
\begin{align}
 |K_1|&\le\frac{Cp^2|s(z)|^4}{n^2}
 \E^{\frac{p-2}p}|\Lambda_n|^p+\frac{Cp^2|s(z)|^2}{n}\E^{\frac{p-1}p}|\Lambda_n|^p.\notag
\end{align}
According to Corollary \ref{eps1} and Lemmas \ref{eps2}, \ref{eps3}, inequality \eqref{etaj}, in the Appendix we have
\begin{equation}\notag
 K_2\le \frac{Cp|s(z)|^3}{n^2}\E^{\frac{p-2}p}|\Lambda_n|^p,
\end{equation}
and
\begin{align}\notag
  K_3&\le\frac{Cp|s(z)|^3}{n^2}\E^{\frac{p-2}p}|\Lambda_n|^p.
\end{align}
To bound $K_4$ we use inequalities \eqref{gg1} and \eqref{lambdanew} and obtain
\begin{align}\label{k4}
 K_4&\le \frac{p^{p-1}|s(z)|^2}n\sum_{j=1}^n\E|\varepsilon_{j\nu}|
 \Big(\frac1n|R_{jj}-s(z)|(1+|\eta_j|)+\frac1n|\eta_j||s(z)|\Big)^{p-1}.
\end{align}
Without loss of generality, we may assume that $p\ge3$ and use the inequality
\begin{equation}
 |\varepsilon_{j\nu}|\le \sqrt n\Big(\frac1n\sum_{l=1}^n|\varepsilon_{l\nu}|^2\Big)^{\frac12}.
\end{equation}
We rewrite now inequality \eqref{k4} in the form
\begin{align}
  K_4&\le \frac{p^{p-1}|s(z)|^2\sqrt n}{n^{p-1}}\E\Big(\frac1n\sum_{j=1}^n|\varepsilon_{j\nu}|^2\Big)^{\frac12}\frac1n\sum_{j=1}^n
 \Big(|R_{jj}-s(z)|(1+|\eta_j|)\notag\\&\qquad\qquad\qquad\qquad\qquad\qquad\qquad\qquad\qquad+|\eta_j||s(z)|\Big)^{p-1}.
\end{align}

Applying H\"older's inequality, we obtain, for $p\ge3$
\begin{align}\label{k42}
 K_4&\le\frac{p^{p-1}|s(z)|^2\sqrt n}{n^{p-1}}\E^{\frac2p}\Big(\frac1n\sum_{j=1}^n|\varepsilon_{j\nu}|^2\Big)^{\frac p2}\frac1n\sum_{j=1}^n
 \E^{\frac{p-1}p}\Big(|R_{jj}-s(z)|(1+|\eta_j|)\notag\\&\qquad\qquad\qquad\qquad\qquad\qquad\qquad\qquad\qquad+|\eta_j||s(z)|\Big)^{p}.
\end{align}
For $p\le3$ we may apply H\"older's inequality directly and obtain
\begin{align}\label{k43}
 K_4&\le \frac{p^{p-1}|s(z)|^2}{n^p}\sum_{j=1}^n\E^{\frac1p}|\varepsilon_{j\nu}|^p
 \Big(\E^{\frac{p-1}p}|R_{jj}-s(z)|^{p}(1+|\eta_j|)^{p}+\E^{\frac{p-1}p}|\eta_j|^p|s(z)|^p\Big).
\end{align}
Inequalities \eqref{k42} and \eqref{k43} together imply
\begin{equation}\label{k44}
 K_4\le \frac{C|s(z)|^{p+1}p^p}{n^p}.
\end{equation}

Collecting the relations \eqref{finish1}, \eqref{finish01}, \eqref{finish001}, \eqref{finish0001}, \eqref{finish00001},  and \eqref{k44}
we get
\begin{align}
 \E|\Lambda_n|^p\le \frac{Cp|s(z)|^3}{n^2}\E^{\frac{p-2}p}|\Lambda_n|^p+\frac{C|s(z)|^{2}}{n}\E^{\frac{p-1}p}|\varphi(\Lambda_n)|^{\frac p{p-1}}
 +\frac{C|s(z)|^{p+1}p^p}{n^p}.
\end{align}

Applying Corollary  \ref{simple} in the Appendix, we get,
\begin{align}\label{disp}
 \E^{\frac1p}|\Lambda_n|^p\le\frac{Cp|s(z)|^{1+\frac1p}}{n}.
\end{align}

Consider now the integral
\begin{equation}\notag
 Int(V)=\int_{-\infty}^{\infty}\E^{\frac1p}|m_n(u+iV)-s(u+iV)|^pdu
\end{equation}
for $V=4$.
Using inequality \eqref{disp}, we have
\begin{align}\notag
 |Int(V)|&\le \frac Cn\int_{-\infty}^{\infty}|s(u+iV)|^{1+\frac1p}du.
\end{align}

Finally, we note that
\begin{equation}\label{finish6}
 \int_{-\infty}^{\infty}|s(z)|^{1+\frac1p}dx\le \int_{-\infty}^{\infty}\int_{-\infty}^{\infty}\frac1{((x-u)^2+V^2)^{\frac{p+1}{2p}}}dudG(x)\le C(p+1).
\end{equation}

Inequalities \eqref{finish1} -- \eqref{finish6}
together imply, that for $p\le C\log n$,
\begin{equation}\label{final!}
 \int_{-\infty}^{\infty}\E^{\frac1p}|\Lambda_n(u+iV)|^pdu\le \frac {C\log^2 n}n.
\end{equation}
\subsection{The bound of the second integral in \eqref{smoothing11}}To finish the proof of Theorem \ref{main} we need to bound the second integral
in \eqref{smoothing1} for $z\in\mathbb G$, $v_0=C_7n^{-1}\log^4 (n+1)$ and $\varepsilon=C_8v_0^{\frac23}$, where the constant $C_8$ is chosen such that so that condition 
\eqref{avcond} holds.
We shall use the results of Theorem \ref{stiltjesmain}.
According to these results we have, for $z\in\mathbb G$,
\begin{equation}\label{jpf}
 J_p(z):=\E^{\frac1p}|m_n(z)-s(z)|^p\le Cp(nv)^{-1}.
\end{equation}
Partition the interval $\mathbb J_{\varepsilon}$ into $k_n=n^4$ subintervals of equal length, that is $-2+\varepsilon=x_0<\cdots<x_{k_n}=2-\varepsilon$.
Note that
\begin{align}
 \sup_{x\in \mathbb J_{\varepsilon}}\Big|&\int_{v_0/\sqrt{\gamma}}^V(m_n(x+iv)-s(x+iv))dv\Big|\notag\\&\le \max_{1\le k\le k_n}\sup_{x_{k-1}\le x\le x_k}
 \Big|\int_{v_0/\sqrt{\gamma}}^V(m_n(x+iv)-s(x+iv))dv\Big|.\notag
\end{align}
Furthermore,
\begin{align}
 \sup_{x_{k-1}\le x\le x_k}
 \Big|\int_{v_0/\sqrt{\gamma}}^V(m_n(x+iv)&-s(x+iv))dv\Big|\notag\\&\le \Big|\int_{v_0/\sqrt{\gamma}}^V(m_n(x_{k-1}+iv)-s(x+iv))dv\Big|\notag\\&+
 \int_{x_{k-1}}^{x_{k}}\int_{v_0/\sqrt{\gamma}}^V|m_n'(x+iv)-s'(x+iv)|dvdx.\notag
\end{align}
Note that, for $z\in\mathbb G$,
\begin{equation}\notag
 |m_n'(x+iv)-s'(x+iv)|\le Cv^{-2}\le Cn^2.
\end{equation}
This yields that
\begin{align}
 \sup_{x_{k-1}\le x\le x_k}&
 \Big|\int_{v_0/\sqrt{\gamma}}^V(m_n(x+iv)-s(x+iv))dv\Big|\notag\\&\le \Big|\int_{v_0/\sqrt{\gamma}}^V(m_n(x_{k-1}+iv)-s(x+iv))dv\Big|+
 Cn^{-1},\notag
\end{align}
and
\begin{align}
 \sup_{x\in \mathbb J_{\varepsilon}}\Big|&\int_{v_0/\sqrt{\gamma}}^V(m_n(x+iv)-s(x+iv))dv\Big|\notag\\&\le \max_{0\le k\le k_n-1}
 \Big|\int_{v_0/\sqrt{\gamma}}^V(m_n(x_k+iv)-s(x_k+iv))dv\Big|+Cn^{-1}.\notag
\end{align}
Using this inequality, we get
\begin{align}
\E\sup_{x\in \mathbb J_{\varepsilon}}&\Big|\int_{v_0/\sqrt{\gamma}}^V(m_n(x+iv)-s(x+iv))dv\Big|^p\notag\\&
\le 2^{p-1}(\sum_{k=0}^{k_n-1}\E\Big|\int_{v_0/\sqrt{\gamma}}^V(m_n(x_k+iv)-s(x_k+iv))dv\Big|^p)+\frac {(2C)^p}{n^p}.\notag
\end{align}
Applying H\"older's inequality,   we get similar to \eqref{integral}
\begin{align}
\E&\Big|\int_{v_0/\sqrt{\gamma}}^V(m_n(x_k+iv)-s(x_k+iv))dv\Big|^p\notag\\&
\le \Big(\int_{v_0/\sqrt{\gamma}}^V\E^{\frac1p}|m_n(x_k+iv)-s(x_k+iv)|^pdv\Big)^p.\notag
\end{align}
Applying now inequality \eqref{jpf} for $p=[ 4\log n]$, we obtain
\begin{align}\label{final!!}
\E^{\frac1{p}}\sup_{x\in \mathbb J_{\varepsilon}}&\Big|\int_{v_0/\sqrt{\gamma}}^V(m_n(x+iv)-s(x+iv))dv\Big|^p\notag\\&
\le Ck_n^{\frac1p}n^{-1}\log^2 n\le Cn^{-1}\log^2 n.
\end{align}
Inequalities \eqref{final!} and \eqref{final!!} complete the proof of Theorem \ref{main}.
\begin{rem}To prove the Corollary \ref{cormain} it is enough to use the results of Corollary \ref{stieltjesmaincor}, 
which imply inequality \eqref{jpf}. Thus Corollary \ref{cormain} is proved.
\end{rem}
\section{Proof of Corollary \ref{stieltjesmaincor}}
We consider the truncated random variables $\widehat X_{jl}$ defined by
\begin{equation}\label{trunc000}
 \widehat X_{jl}:=X_{jl}\mathbb I\{|X_{jl}|\le cn^{\frac14} \}.
\end{equation}
Let $\widehat {\mathcal F}_n(x)$ denote the empirical spectral distribution function of the matrix $\widehat{\mathbf W}=\frac1{\sqrt n}(\widehat X_{jl})$.
\begin{lem}\label{trunc}
 Assuming the conditions of Theorem \ref{main} there exist constants $C, c>0$ such that, for any $p\ge1$
 \begin{equation}\notag
\E^{\frac1p}|m_n(z)-s(z)|^p\le \frac{Cp}{nv}.
 \end{equation}

\end{lem}
\begin{proof}
 We use rank the inequality of Bai. See \cite{BaiSilv:2010},  Theorem A.43, p. 503.
 According to this inequality
 \begin{equation}\notag
 |m_n(z)-s(z)| \le \frac 1{nv}\text{\rm rank}(\mathbf X-\widehat{\mathbf X}).
 \end{equation}
Using the obvious fact that the  rank of a matrix is not  larger then the number of its non-zero entries, we may write
\begin{align}
 \E|m_n(z)-s(z)|^p &\le \frac 1{(nv)^p}\E\Big(\sum_{j,k=1}^n\mathbb I\{|X_{jk}|\ge Cn^{\frac14}\}\Big)^p\notag\\&\le 
 \frac {2^p}{(nv)^p}\Big(\big(\sum_{j,k=1}^n\E\mathbb I\{|X_{jk}|\ge Cn^{\frac14}\}\big)^p\notag\\&\qquad\qquad+\E\Big|\sum_{j,k=1}^n(\mathbb I\{|X_{jk}|\ge Cn^{\frac14}\}
 -\E\mathbb I\{|X_{jk}|\ge Cn^{\frac14}\}\big)\Big|^p\Big).\notag
\end{align}
Applying Chebyshev's and Rosenthal's inequalities, we get
\begin{align}
 \E|m_n(z)-s(z)|^p &\le\frac {2^p}{(nv)^p}\Big(\big(\frac{1}{n^2}\sum_{j,k=1}^n\E X_{jk}^8\big)^p\notag\\&+C^pp^p\big(\frac{1}{n^2}\sum_{j,k=1}^n\E X_{jk}^8)^{\frac p2}+
 \frac1{n^2}\sum_{j,k=1}^n\E X_{jk}^8\big)\Big)\le \frac {(Cp)^p}{(nv)^p}.\notag
\end{align}

Thus, the Lemma is proved.
\end{proof}
Introduce now $\widetilde X_{jk}=\widehat X_{jk}-\E\widehat X_{jk}$ and $\widetilde{\mathbf W}=\frac1{\sqrt n}(\widetilde X_{jk})_{j,k=1}^n$.
Denote by $\widetilde m_n(z)$ the Stieltjes transform of empirical distribution function of the matrix $\widetilde{\mathbf W}$ and let
$\widehat m_n(z)$ denote the Stieltjes transform of the matrix
$\widehat{\mathbf W}$.
Furthermore, we re-normalize the matrix $\widetilde {\mathbf W}$. Let $\sigma_{jk}^2=\E|\widetilde X_{jk}|^2$. We  introduce the random variables
$\breve X_{jk}=\sigma_{jk}^{-1}\widetilde X_{jk}$. And let $\hat m_n(z)$ denote the Stieltjes transform of the empirical spectral distribution function of the matrix 
$\breve{\mathbf W}=
\frac1{\sqrt n}(\breve   X_{jk})_{j,k=1}^n$.
\begin{lem}\label{trunc2}
\begin{equation}\notag
 \E|\widetilde m_n(z)-\breve m_n(z)|^p\le \frac{Cp}{(nv)^{\frac32}}.
 \end{equation}
\end{lem}
\begin{proof}
 Using the resolvent equality \eqref{reseq}, we get
 \begin{equation}\notag
  \widetilde m_n(z)-\breve m_n(z)=\frac1n\Tr\widetilde{\mathbf R}(\breve{\mathbf W}-\widetilde{\mathbf W})\breve{\mathbf R}
 \end{equation}
Using the obvious inequalities $|\Tr\mathbf A\mathbf B|\le \|\mathbf A\|_2\|\mathbf B\|_2$ \newline
and $\|\mathbf A\mathbf B\|_2\le\|\mathbf A\|\|\mathbf B\|_2$, we obtain
\begin{align}\label{breve}
 |\widetilde m_n(z)-\breve m_n(z)|\le \frac1n\|\widetilde{\mathbf R}\|\|\breve{\mathbf R}\|_2\|\breve{\mathbf W}-\widetilde{\mathbf W}\|_2
\end{align}
Note that
\begin{equation}\label{w1}
\|\breve{\mathbf W}-\widetilde{\mathbf W}\|_2^2=\frac1n\sum_{j,k=1}^n(1-\sigma_{jk})^2{\breve X}_{jk}^2.
 \end{equation}
 Furthermore, we observe that
 \begin{equation}\label{w2}
  (1-\sigma_{jk})^2\le (1-\sigma_{jk}^2)^2\le (\E X_{jk}^2\mathbb I\{|X_{jk}|\ge cn^{\frac14}\})^2\le C\mu_8^2n^{-3}.
 \end{equation}
Relations \eqref{w1} and \eqref{w2} together imply
\begin{equation}\notag
\|\breve{\mathbf W}-\widetilde{\mathbf W}\|_2^2\le Cn^{-3}\|\breve{\mathbf W}\|_2^2.
\end{equation}

Note that the $\breve X_{jl}$ satisfy the condition
 \begin{equation}\label{trunc21}
  |\breve X_{jl}|\le Dn^{\frac14}, \quad\E \breve X_{jl}=0 \text{ and }\E {\breve X}_{jk}^2=1,
 \end{equation}
 for some absolute constant $D$. We may apply Theorem \ref{stiltjesmain}. According to this theorem we have, for $q\le C\log n$,
 \begin{equation}\label{imr1}
  \E|m_n(z)|^q\le C^q.
 \end{equation}
Furthermore, we note that, by Lemma \ref{resol00} in the Appendix
\begin{equation}\label{imr}
 \frac1n\|\breve{\mathbf R}\|_2^2= v^{-1}\im \breve m_n(z)\le v^{-1}|\breve m_n(z)|.
\end{equation}
Inequality \eqref{breve} yields
\begin{equation}\notag
 \E|\widetilde m_n(z)-\breve m_n(z)|^p\le v^{-p}n^{-\frac{3p}2}\E^{\frac12}(n^{-\frac12}\|\breve{\mathbf R}\|_2)^{2p}\E^{\frac12}(n^{-\frac12}\|\breve{\mathbf W}\|_2)^{2p}
\end{equation}
Applying inequalities \eqref{imr1} and \eqref{imr}, we get
\begin{equation}\label{emn}
 \E|\widetilde m_n(z)-\breve m_n(z)|^p\le C^pn^{-\frac{3p}2}v^{-\frac{3p}2}\E^{\frac12}(n^{-\frac12}\|\breve{\mathbf W}\|_2)^{2p}.
\end{equation}
To bound the last factor in the r.h.s. of \eqref{emn} we use standard arguments based on Rosenthal's inequality.
We may write
\begin{align}\label{enm1}
 \E(n^{-1}\|\breve{\mathbf W}\|_2^2)^{p}&=\frac1{n^{2p}}\E(\sum_{j,k}{\breve X}_{jk}^2)^p\le \frac{2^p}{n^{2p}}\Big(\sum_{j,k}\E{\breve X}_{jk}^2\Big)^p\notag\\&+
 \frac{C^pp^p}{n^{2p}}\Big(\big(\sum_{j,k}\E({\breve X}_{jk}^2-1)^2\big)^{\frac p2}+\sum_{j,k=1}^n\E|{\breve X}_{jk}^2-1|^p\Big)\le C^pp^p
\end{align}

Using now inequalities \eqref{enm1} and \eqref{emn}, we get the claim.
Thus, Lemma \ref{trunc2} is proved.
\end{proof}
\begin{lem}\label{trunc1}
\begin{equation}\notag
 \E|\widetilde m_n(z)-\widehat m_n(z)|\le \frac{C\mu_8}{n^{\frac32}v^{\frac32}}.
 \end{equation}
\end{lem}
\begin{proof}
According to the resolvent equality \eqref{reseq}, we have
\begin{equation}\notag
 \widetilde m_n(z)-\widehat m_n(z)=\frac1n\Tr(\widetilde{\mathbf R}-\widehat{\mathbf R})
 =\frac1n\Tr (\widetilde{\mathbf W}-\widehat{\mathbf W})\widetilde{\mathbf R}\widehat{\mathbf R}.
\end{equation}
Similar to \eqref{breve} we get
\begin{equation}\label{dif1}
|\widetilde m_n(z)-\widehat m_n(z)|\le n^{-1}\|\widehat{\mathbf R}\|\|\widetilde{\mathbf R}\|_2\|\E\widehat{\mathbf W}\|_2.
\end{equation}
Furthermore, we note that
\begin{align}\notag
 |\E \widehat X_{jk}|&\le Cn^{-\frac 74}\mu_8.
\end{align}
This yields
\begin{align}\notag
 n^{-\frac12}\|\E\widehat{\mathbf W}\|_2\le Cn^{-\frac 94}\text{ a. s.}
\end{align}

By Lemma \ref{trunc2}, we have
\begin{equation}\label{dif2}
 \E|\widetilde m_n(z)|^p\le C^p.
\end{equation}
This implies that
\begin{equation}\label{dif3}
\E(\frac1{\sqrt n}\|\widetilde {\mathbf R}\|_2)^p\le C^p v^{-\frac p2} 
\end{equation}
Combining now inequalities \eqref{dif1}, \eqref{dif2} and \eqref{dif3}, we get
\begin{equation}\label{dif4}
\E|\widetilde m_n(z)-\widehat m_n(z)|^p\le\frac {C\mu_8}{n^{\frac{9p}4}v^{\frac{3p}2}}\le \frac {C^p}{n^{\frac{3p}2}v^{\frac{3p}2}}.
\end{equation}
Thus Lemma \ref{trunc1} is proved.
\end{proof}
Lemmas \ref{trunc}, \ref{trunc2}, \ref{trunc1} together imply the result of Corollary \ref{stieltjesmaincor}.
Thus Corollary \ref{stieltjesmaincor} is proved.
\section{Proof of Theorem \ref{stiltjesmain}} The main problem in proving Theorem \ref{stiltjesmain} is the the derivation of the following bound 
\begin{equation}\notag
 \E|R_{jj}|^p\le C^p,
\end{equation}
for $j=1,\ldots,n$ and any $z\in\mathbb G$.
This bound was shown in \cite{GT:2014}.
To prove this bound we used an approach similar to that of Lemma 3.4 in \cite{SchleinMaltseva:2013}. 
We succeeded in the case of finite moments only developing  new  bounds of quadratic forms of the following type
\begin{equation}\notag
 \E|\frac1n\sum_{l\ne k}X_{jl}X_{jk}R^{(j)}_{kl}|^p\le\left(\frac {Cp}{\sqrt {nv}}\right)^p. 
\end{equation}
These estimates are based on a recursive scheme of using  Rosenthal's and Burkholder's inequalities.
\subsection{The Key Lemma}\label{key} In this Section we state  auxiliary lemmas needed for the proof of Theorem \ref{stiltjesmain},
which have been  proved in \cite{GT:2014}.
Recall that the
Stieltjes transform of an empirical  spectral distribution function $\mathcal F_n(x)$, say $m_n(z)$, is given by
\begin{equation}\label{trace}
m_n(z)=\frac1n\sum_{j=1}^n
R_{jj}=\frac1{n}\Tr\mathbf  R.
\end{equation}
 (see, for instance, equality (4.3)
in \cite{GT:03}).

For any $\mathbb J\subset T$ denote $\mathbb T_{\mathbb J}=\mathbb T\setminus\mathbb J$.
For any $\mathbb J\subset \mathbb T$ and $j\in\mathbb T_{\mathbb J}$ define  the quadratic form,
$$
Q^{(\mathbb J,j)}:=\frac1n\sum_{l\in\mathbb T_{\mathbb J}}\Big|\sum_{r\in \mathbb T_{\mathbb J}\cap\{1,\ldots,l-1\}}X_{jl}R^{(\mathbb J,j)}_{kl}\Big|^2
$$
and
$$
\widetilde Q^{(\mathbb J,j)}:=\frac1n\sum_{l\in\mathbb T_{\mathbb J}}\Big|\sum_{r\in \mathbb T_{\mathbb J}\cap\{1,\ldots,l-1\}}X_{jl}[(\mathbf R^{(\mathbb J,j)})^2]_{kl}\Big|^2.
$$

\begin{lem}\label{bp1} Assuming the conditions of Theorem \ref{main} there exist  constants $A_1, C, C_3$ depending on $\mu_4$ and $D$ only such that 
we have for $v\ge v_0$  and  $p\le A_1(nv)^{\frac14}$ and for any $\mathbb J\subset\mathbb T$ such that $|\mathbb J|\le C\log n$,

\begin{align}\label{q1}
 \E (Q^{(\mathbb J,j)})^p\le(C_3p)^{2p}v^{-p}.
\end{align}

\end{lem}
\begin{cor}\label{q1*}Assuming the conditions of Theorem \ref{main} and for $z=u+iV$ with $V=4$, we have
 \begin{align}\notag
 \E (Q^{(\mathbb J,j)})^p\le C^pp^{2p}.
\end{align}

\end{cor}
\begin{proof}
 The result immediately follows from Lemma \ref{bp1}
\end{proof}

\begin{Proof of} {\it Lemma \ref{bp1}}. For the proof of Lemma \ref{bp1} see \cite{GT:2014}, Lemma 5.4, Section 5.
 %
\end{Proof of}
\begin{lem}\label{bp1*} Assuming the conditions of Theorem \ref{main} there exist  constants $A_1, C, C_3$ depending on $\mu_4$ and $D$ only such that 
we have for $v\ge v_0$  and  $p\le A_1(nv)^{\frac14}$ and for any $\mathbb J\subset\mathbb T$ such that $|\mathbb J|\le C\log n$,

\begin{align}\label{q1**}
 \E (\widetilde Q^{(\mathbb J,j)})^p\le(C_3p)^{2p}v^{-3p}.
\end{align}

\end{lem}
\begin{Proof of} {\it Lemma \ref{bp1}}. The proof of Lemma \ref{bp1*} is similar to the proof of  \cite[Lemma 5.4]{GT:2014}.
 
\end{Proof of}
\subsection{Diagonal Entries of the Resolvent Matrix}\label{diag}

Recall that
\begin{equation}\label{repr001}
 R_{jj}=-\frac1{z+m_n(z)}+\frac1{z+m_n(z)}\varepsilon_jR_{jj},
\end{equation}
or
\be\label{repr01}
 R_{jj}=-\frac1{z+s(z)}+\frac{\Lambda_nR_{jj}}{(z+s(z))}+
\frac1{z+s(z)}\varepsilon_jR_{jj}, \en
where
$\varepsilon_j:=\varepsilon_{j1}+\varepsilon_{j2}+\varepsilon_{j3}+
\varepsilon_{j4}$  with
\begin{align}
\varepsilon_{j1}&:=\frac1{\sqrt n}X_{jj},\quad
\varepsilon_{j2}:=-\frac1n{\sum_{k\ne l
\in\mathbb T_j}}X_{jk}X_{jl}R^{(j)}_{kl},\quad
\varepsilon_{j3}:=-\frac1n{\sum_{k\in\mathbb T_j}}(X_{jk}^2-1)R^{(j)}_{kk},\notag\\
\varepsilon_{j4}&:=\frac1n(\Tr \mathbf R-\Tr\mathbf R^{(j)}),\quad\notag
\Lambda_n:=m_n(z)-s(z)=\frac1n\Tr\mathbf R-s(z),\\
\varepsilon_{j4}&=\Lambda_n -\Lambda_n^{(j)}.\label{epsjn}
\end{align}


\begin{cor}\label{cor8} Assuming the  conditions of Theorem \ref{main}, for all $A_1>0$ there exists a positive constant $A_0=A_0(A_1,\mu_4,D)$ 
depending on $\mu_4, D$ and $A_1$
  such that, for $p\le A_1(nv)^{\frac14}$ and $v\ge v_0= A_0n^{-1}\log^4 n$
there exist an absolute  constants $C_0>0$ such that 
\begin{equation}\label{r11}
\E|R_{jj}|^p\le C_0^p,
\end{equation}
and
\begin{equation}\label{r12}
 \E\frac1{|z+m_n(z)|^p}\le C_0^p
\end{equation}

\end{cor}

\begin{proof}
For the proof of this Corollary see \cite{GT:2014}, Section 6, Corollary 6.10.
\end{proof}

\section{Estimation of $\E| m_n(z)-s(z)|^p$.}\label{expect}
We return now to the representation \eqref{repr01} which may be rewritten as
\begin{align}\label{lambda}
m_n(z)&=\frac1n\sum_{j=1}^n R_{jj}=-\frac1{z+m_n(z)}+\frac{T_n(z)}{z+s(z)+m_n(z)}.\end{align}
We develop the last equality as follows
\begin{align}\label{eq00}
m_n(z)=s(z)+\frac1n\sum_{j=1}^n\frac{\varepsilon_{j4}R_{jj}}{z+s(z)+m_n(z)}+
\E\frac{\widehat T_n(z)}{z+s(z)+m_n(z)},
\end{align}
where 
$$
\widehat T_n=\sum_{\nu=1}^3\frac1n\sum_{j=1}^n\varepsilon_{j\nu}R_{jj}.
$$
Note that, by equality \eqref{3.9},
\begin{equation}\notag
\frac1n\sum_{j=1}^n\varepsilon_{j4}R_{jj}=-\frac1n\frac{d m_n(z)}{dz}.
\end{equation}
Furthermore, we write
\begin{align}\label{represent1}
\Lambda_n=-\frac1n\frac{ m_n'(z)}{z+s(z)+m_n(z)}+
\frac{\widehat T_n(z)}{z+s(z)+m_n(z)}.
\end{align}
We shall investigate the quantity
\begin{equation}\notag
 J_p:=\E|m_n(z)-s(z)|^p=\E|\Lambda_n|^p,\text{ for } p\ge2.
\end{equation}
We introduce the  notation
\begin{equation}\notag
 \varphi_p(z)=\overline z|z|^{p-2}.
\end{equation}
In these terms we may represent $J_p$ as
\begin{equation}\notag
 J_p=\E\Lambda_n\varphi_p(\Lambda_n)
\end{equation}
and expand this equality using the  representation \eqref{represent1} arriving at
\begin{align}\notag
 J_p=-\frac1n\E\frac{ m_n'(z)\varphi_p(\Lambda_n)}{z+s(z)+m_n(z)}+\E\frac{\widehat T_n(z)\varphi_p(\Lambda_n)}{z+s(z)+m_n(z)}.
\end{align}

Denote by
\begin{align}
\mathfrak T_1&=-\frac1n\E\frac{ m_n'(z)\varphi_p(\Lambda_n)}{z+s(z)+m_n(z)},\notag\\
\mathfrak T_2&=\frac{\widehat T_n\varphi_p(\Lambda_n)}{z+s(z)+m_n(z)}.\notag
\end{align}
This is an approach similar to  that used by us in the proof of Lemma 6.1  \cite{GT:2003}.

\subsection{Estimation of $\mathfrak T_1$} Using Lemma \ref{resol00} in the Appendix, we get
\begin{equation}\label{fini1}
 |T_1|\le \frac C{nv}\E|\varphi_p(\Lambda_n)|\le \frac1{nv} J_p^{\frac{p-1}p}.
\end{equation}

\subsection{Estimation  of $\mathfrak T_{2}$}
The quantity $\mathfrak T_{2}$ we represent in the form
\begin{align}
\mathfrak T_{2}=\mathfrak T_{21}+\mathfrak T_{22}+\mathfrak T_{23},\notag
\end{align} 
where, for $\nu=1,2,3$,
\begin{align}
\mathfrak T_{2\nu}&=-\frac1n\sum_{j=1}^n\E\frac{\varepsilon_{j\nu}R_{jj}\varphi(\Lambda)}{z+m_n(z)+s(z)}.\notag
\end{align}
\subsubsection{Estimation of $\mathfrak T_{21}$}
We represent $\mathfrak T_{21}$ in the form
\begin{equation}\notag
 \mathfrak T_{21}=L_1+L_2,
\end{equation}
where
\begin{align}
 L_1&=\E\frac{(\frac1n\sum_{j=1}^n\varepsilon_{j1})\frac1{z+m_n(z)}\varphi(\Lambda_n)}{z+m_n(z)+s(z)}\notag\\
 L_2&=-\frac1n\sum_{j=1}^n\E\frac{\varepsilon_{j1}(R_{jj}+\frac1{z+m_n(z)})\varphi(\Lambda_n)}{z+m_n(z)+s(z)}.\notag
\end{align}
We first consider the term $L_1$.
Applying H\"older's inequality and Lemma \ref{lem00} in the Appendix, we get
\begin{align}
 |L_1|&\le E^{\frac1p}\frac{\Big|\frac1{n\sqrt n}\sum_{j=1}^nX_{jj}\Big|^p}{|z+m(z)|^p|z+s(z)+m_n(z)|^p}
 \E^{\frac{p-1}p}|\varphi(\Lambda_n)|^{\frac p{p-1}}\notag\\&\le\frac1{\sqrt{|z^2-4|}}
 E^{\frac1{2p}}|\frac1{n\sqrt n}\sum_{j=1}^nX_{jj}|^{2p}\E^{\frac1{2p}}\frac1{|z+m_n(z)|^{2p}} \E^{\frac{p-1}p}|\varphi(\Lambda_n)|^{\frac p{p-1}}.\notag
\end{align}
Using now Corollary \ref{cor8} and the inequality $|\varphi(\Lambda_n)|^{\frac p{p-1}}\le |\Lambda_n|^p$, we get
\begin{align}
 |L_1|\le\frac{C_0}{\sqrt{|z^2-4|}}E^{\frac1{2p}}\Big|\frac1{n\sqrt n}\sum_{j=1}^nX_{jj}\Big|^{2p}J_p^{\frac{p-1}p}.\notag
\end{align}
Applying Rosenthal's inequality to the sum $\frac1{n\sqrt n}\sum_{j=1}^nX_{jj}$, we obtain, for $z\in\mathbb G$
\begin{equation}\label{new00}
 |L_1|\le\frac{C_0p}{n\sqrt{|z^2-4|}}\E^{\frac{p-1}p}|\Lambda_n|^p\le \frac{C_0p}{nv}J_p^{\frac{p-1}p}.
\end{equation}

Using the representation \eqref{repr01}, we get
\begin{align}
L_2&=-\frac1n\sum_{j=1}^n\E\frac{\varepsilon_{j1}\varepsilon_j\varphi(\Lambda_n)R_{jj}}{(z+m_n(z))(z+m_n(z)+s(z))}.\notag
\end{align}
This representation yields using
$\eps_{j1}\eps_{j\nu}\le (\eps_{j1}^2 + \eps_{j\nu}^2)/2$,
\begin{align}\label{lada00}
 |L_2|\le \sum_{\mu=1}^4\frac2n\sum_{j=1}^n\E\frac{|\varepsilon_{j\mu}|^2|R_{jj}||\varphi(\Lambda_n)|}{|z+m_n(z)||z+m_n(z)+s(z)|}=:\sum_{\mu=1}^4L_{2\mu}.
\end{align}
First we bound $L_{2\mu}$ for $\mu=1$.
By definition of $\varepsilon_{j1}$, we may write
\begin{align}
 |L_{21}|&\le \frac2{n^2}\sum_{j=1}^n\E\frac{|X_{jj}|^2|R_{jj}||\varphi(\Lambda_n)|}{|z+m_n(z)|^2|z+m_n(z)+s(z)|}.\notag
\end{align}
Applying H\"older's inequality and Lemma \ref{lem00} in the Appendix, we get
\begin{align}\label{new1}
|L_{21}|&\le\frac2{n^2\sqrt{|z^2-4|}}
\sum_{j=1}^n\E^{\frac1{2p}}|R_{jj}|^{2p}
\E^{\frac1{2p}}\frac1{|z+m_n(z)|^{2p}}\notag\\&\qquad\qquad\qquad\qquad\times
\E^{\frac{p-1}p}|X_{jj}|^{\frac{2p}{p-1}}
|\varphi(\Lambda_n)|^{\frac{p}{p-1}}.
\end{align}
Furthermore, we observe that, $\frac{2p}{p-1}\le 4$ for $p\ge2$ and by Lemma \ref{exp*} with $\zeta=\varepsilon_{j4}$, we get
\begin{align}
\E^{\frac{p-1}p}|X_{jj}|^{\frac{2p}{p-1}}
|\varphi(\Lambda_n)|^{\frac{p}{p-1}}
&\le \E^{\frac{p-1}p}|X_{jj}|^{\frac{2p}{p-1}}
|\Lambda_n|^p\notag\\&\le\Big(\E^{\frac{p-1}p}|X_{jj}|^{\frac{2p}{p-1}}(\text{\rm e}
|\Lambda_n^{(j)}|^p+\frac{(p+1)^p}{(nv)^p})\Big)\notag\\&
\le \text{\rm e}\mu_4^{\frac12}\E^{\frac{p-1}p}|\Lambda_n^{(j)}|^p+\frac{(p+1)^{p-1}}{(nv)^{p-1}}\Big).
\end{align}
 Applying the inequality $(a+b)^p\le \text{\rm e}a^p+(p+1)^pb^p$, we get
\begin{equation}
 \E|\varphi(\Lambda_n^{(j)}|^p\le (\text{\rm e}\E|\varphi(\Lambda_n)|^p+(p+1)^p\E|\varphi(\Lambda_n)-\varphi(\Lambda_n^{(j)})|^p).
\end{equation}

Using Corollary \ref{cor8} and Lemmas \ref{exp0} and \ref{exp*} in the Appendix, we obtain 
\begin{align}\label{new1*}
|L_{21}|&\le \frac{Cp}{\sqrt{|z^2-4|}\,n}J_p^{\frac{p-1}p}
 +\frac{(Cp)^{p-1}}{n^pv^{p-1}}.
\end{align}
Applying that $\sqrt{|z^2-4|}\ge Cv$ for $z\in\mathbb G$, we get
\begin{align}\label{new1**}
|L_{21}|&\le
 \frac{Cp}{nv}
J_p^{\frac{p-1}p}
 +\frac{(Cp)^{p-1}}{n^pv^{p-1}}.
\end{align}

Consider now $L_{2\mu}$ for $\mu=2,3$.
Recall that 
\begin{align}\label{new6}
L_{2\mu}=\frac{2}n\sum_{j=1}^n\E\frac{ |\varepsilon_{j\mu}|^2||R_{jj}|\varphi(\Lambda_n)|}{|z+m_n(z)||z+m_n(z)+s(z)|}.
\end{align}
Using H\"older's inequality , we may obtain, for $z\in\mathbb G$,
\begin{align}\label{new7}
 |L_{2\mu}|\le \frac{2} n\sum_{j=1}^n\E^{\frac1{2p}}|R_{jj}|^{2p}\E^{\frac1{2p}}\frac1{|z+m_n(z)|^{2p}}\E^{\frac{p-1}p}
 \frac{|\varepsilon_{j\mu}|^{\frac{2p}{p-1}}}{|z+m_n(z)+s(z)|^{\frac{p}{p-1}}}|\Lambda_n|^p.
\end{align}
Using Corollary \ref{cor8}, Lemma \ref{exp*} with $\zeta=\varepsilon_{j4}$, and Lemma \ref{basic8}, we get
 \begin{align}\label{new7*}
 |L_{2\mu}|&\le \frac{C}n\sum_{j=1}^n\E^{\frac{p-1}p}
 \frac{|\varepsilon_{j\mu}|^{\frac{2p}{p-1}}}{|z+m_n(z)+s(z)|^{\frac{p}{p-1}}}|\Lambda_n^{(j)}|^p\notag\\&\qquad\qquad\qquad+\frac{(cp)^{p-1}}{(nv)^{p-1}n}\sum_{j=1}^n\E^{\frac{p-1}p}
 \frac{|\varepsilon_{j\mu}|^{\frac{2p}{p-1}}}{|z+m_n(z)+s(z)|^{\frac{p}{p-1}}}.
\end{align}
Furthermore, we use inequality \eqref{lar1} in the Appendix. We get, for $z\in\mathbb G$,
\begin{align}\label{lada01}
 |L_{2\mu}|&\le \frac{C}n\sum_{j=1}^n\E^{\frac{p-1}p}
 \frac{|\varepsilon_{j\mu}|^{\frac{2p}{p-1}}}{|z+m_n^{(j)}(z)+s(z)|^{\frac{p}{p-1}}}|\Lambda_n^{(j)}|^p\notag\\&\qquad +\frac{(cp)^{p-1}}{(nv)^{p-1}n}\sum_{j=1}^n\E^{\frac{p-1}p}
 \frac{|\varepsilon_{j\mu}|^{\frac{2p}{p-1}}}{|z+m_n^{(j)}(z)+s(z)|^{\frac{p}{p-1}}}.
\end{align}
Conditioning on $\mathfrak M^{(j)}$ and applying  H\"older's inequality, we get
\begin{align}\label{lada1}
|L_{2\mu}|&\le \frac{2}n\sum_{j=1}^n\E^{\frac{p-1}p}\Big(\E^{\frac p{2(p-1)}}\{
 \frac{|\varepsilon_{j\mu}|^{4}}{|z+m_n^{(j)}(z)+s(z)|^2}\Big|\mathfrak M^{(j)}\}\Big)
 |\Lambda_n^{(j)}|^p\notag\\&\qquad\qquad\qquad+\frac{(cp)^{p-1}}{(nv)^{p-1}n}\sum_{j=1}^n\E^{\frac12}
 \frac{|\varepsilon_{j\mu}|^{4}}{|z+m_n^{(j)}(z)+s(z)|^2}. 
\end{align}
Inequality  \eqref{lada1}, and Corollary \ref{corgot} together imply that, for $z\in\mathbb G$ and for $\mu=2,3$,
\begin{align}\label{lada2}
 |L_{2\mu}|&\le \frac{Cp}{nv}J_p^{\frac{p-1}p}+\frac{(cp)^{p-1}}{(nv)^{p}}. 
\end{align}
Finally, we observe that
\begin{align}
 |L_{24}|\le \frac{Cp}{n^2v^2\sqrt{|z^2-4|}}J_p^{\frac{p-1}p}\le \frac{Cp}{nv}J_p^{\frac{p-1}p}.\notag
\end{align}

Combining inequalities \eqref{lada00}, \eqref{new1**}, \eqref{lada2}, we obtain
\begin{align}\label{new10}
 |L_2|\le\frac{Cp}{nv}J_p^{\frac{p-1}p}+\frac{(cp)^{p-1}}{(nv)^{p}}.
\end{align}
Inequalities \eqref{new10} and \eqref{new00} together imply
\begin{equation}\label{t21}
 |\mathfrak T_{21}|\le \frac{Cp}{nv}J_p^{\frac{p-1}p}+\frac{C^pp^{p-1}}{(nv)^p}.
\end{equation}
\subsection{Estimation of $\mathfrak T_{2\nu}$, for $\nu=2,3$} Recall that
\begin{align}
\mathfrak T_{2\nu}=-\frac1n\sum_{j=1}^n\E\frac{\varepsilon_{j\nu}R_{jj}\varphi(\Lambda_n)}{z+m_n(z)+s(z)}.
\end{align}

 Recall that
\begin{equation}\notag
 \varepsilon_{j4}=\frac1n(\Tr \mathbf R-\Tr \mathbf R^{(j)})=\frac1n(1+\eta_{j0}+\eta_j)R_{jj},
\end{equation}
See relations \eqref{important} -- \eqref{important2}.

\begin{equation}\notag
  \widetilde{\Lambda}_n^{(j)}=\frac1n\Tr\mathbf R^{(j)}+\frac{s(z)}n(1+\eta_{j0}).
\end{equation}
Similarly as in the  Section \ref{firsttypeint} we represent $\mathfrak T_{2\nu}$ in the form
\begin{equation}\notag
 \mathfrak T_{2\nu}=M_1+M_2+M_3+M_4,
\end{equation}
where
\begin{align}
 M_1&=\frac1n\sum_{j=1}^n\E\frac{\varepsilon_{j\nu}\frac1{z+m_n^{(j)}(z)}\varphi(\widetilde{\Lambda}_n^{(j)})}{z+m_n^{(j)}(z)+s(z)},\notag\\
 M_2&=-\frac1n\sum_{j=1}^n\E\frac{\varepsilon_{j\nu}(R_{jj}+\frac1{z+m_n^{(j)}(z)})\varphi(\Lambda_n)}{z+m_n(z)+s(z)},\notag\\
 M_3&=\frac1n\sum_{j=1}^n\E\frac{\varepsilon_{j\nu}\frac1{z+m_n^{(j)}(z)}(\varphi(\Lambda_n)-\varphi(\widetilde{\Lambda}_n^{(j)}))}{z+m_n(z)+s(z)},\notag\\
 M_4&=-\frac1n\sum_{j=1}^n\E\frac{\varepsilon_{j\nu}\frac1{z+m_n^{(j)}(z)}\varphi(\widetilde{\Lambda}_n^{(j)})\varepsilon_{j4}}{(z+m_n^{(j)}(z)+s(z))(z+m_n(z)+s(z))}.
\end{align}
Note that, by independence of $X_{jk}, k\in\mathbb T$  and $\mathfrak M^{(j)}$,
\begin{equation}\label{m1}
 M_1=0.
\end{equation}

Furthermore, we represent
\begin{equation}\notag
M_{2}=M_{21}+M_{22}+M_{23},
\end{equation}
where, for $\mu=1,2,3$
\begin{align}\notag
M_{2\mu}&=\frac1n\sum_{j=1}^n\E\frac{{\varepsilon}_{j\nu}\varepsilon_{j\mu}R_{jj}\varphi(\Lambda_n)}{(z+m^{(j)}(z))(z+s(z)+m_n(z))}.
\end{align}
Using the inequality $|ab|\le \frac12(a^2+b^2)$, we obtain, for $\nu=2,3$, and $\mu=1,2,3$
\begin{align}\notag
|M_{2\mu}|&\le \frac1n\sum_{j=1}^n\E\frac{(|\varepsilon_{j\nu}|^2+|\varepsilon_{j\mu}|^2)|R_{jj}||\varphi(\Lambda_n)|}{(z+m^{(j)}(z))(z+s(z)+m_n(z))}.
\end{align}

Similar to inequalities  \eqref{lada01}, \eqref{lada1}, we get
\begin{align}\label{lada3}
 |M_{2}|\le \frac{Cp}{nv}J^{\frac{p-1}p}+\frac{C^pp^{p-1}}{(nv)^p}.
\end{align}

\subsection{Estimation of $M_3$}Recall that
\begin{align}\notag
 M_3&=\frac1n\sum_{j=1}^n\E\frac{\varepsilon_{j\nu}\frac1{z+m_n^{(j)}(z)}(\varphi(\Lambda_n)-\varphi(\widetilde{\Lambda}_n^{(j)}))}{z+m_n(z)+s(z)}.
\end{align}
Let 
\begin{align}\label{delta}
     \delta_j=\Lambda_n-\widetilde{\Lambda}_n^{(j)}=\frac1n(R_{jj}-s(z))(1+\eta_{j0})+\frac1n\eta_jR_{jj}.
    \end{align}

Applying Taylor's formula, we represent it in the form
\begin{equation}\label{m3}
M_3=M_{31}+M_{32},
\end{equation}
where
\begin{align}
M_{31}&=\frac{1}{n^2}\sum_{j=1}^n\E\frac{\varepsilon_{j\nu}(R_{jj}-s(z))(1+\eta_{j0})\varphi'(\widetilde\Lambda_n^{(j)}+\tau\delta_j)}{(z+m_n^{(j)}(z))(z+m_n(z)+s(z))},\notag\\
M_{32}&=\frac{1}{n^2}\sum_{j=1}^n\E\frac{\varepsilon_{j\nu}\eta_jR_{jj}\varphi'(\widetilde\Lambda_n^{(j)}+\tau\delta_j)}{(z+m_n(z)+s(z))(z+m_n^{(j)}(z))}.
\end{align}

\subsubsection{Estimation of $M_{31}$}
First we note that
\begin{equation}\label{mn'}
R_{jj}-s(z)=\frac{\Lambda_n^{(j)}s(z)}{z+m_n^{(j)}(z)}+\sum_{\mu=1}^3\frac{\varepsilon_{j\mu}}{z+m_n^{(j)}(z)}R_{jj}.
\end{equation}

We represent now $M_{31}$ in the form
\begin{align}\label{m31p}
 M_{31}=G_1+G_2+G_{3}+G_{4},
\end{align}
where
\begin{align}
G_1&=\frac{1}{n^2}\sum_{j=1}^n\E\frac{s(z)\varepsilon_{j\nu}(1+\eta_{j0})\Lambda_n^{(j)}\varphi'(\widetilde{\Lambda}_n^{(j)}+\tau\delta_j)}
{(z+m_n^{(j)}(z))^2(z+m_n(z)+s(z))},\notag\\
G_{\mu+1}&=\frac{1}{n^2}\sum_{j=1}^n\E\frac{\varepsilon_{j\nu}\varepsilon_{j\mu}(1+\eta_{j0})R_{jj}\varphi'(\widetilde{\Lambda}_n^{(j)}+\tau\delta_j)}
{(z+m_n^{(j)}(z))^2(z+m_n(z)+s(z))},\text{ for }\mu=1,2,3.\notag
\end{align}
We continue with $G_1$, applying Lemma \ref{exp*}.
We get
\begin{align}
 |G_1|\le \frac{Cp}{n^2}\sum_{j=1}^n\E\frac{|\varepsilon_{j\nu}||1+\eta_{j0}||\Lambda_n^{(j)}||\widetilde{\Lambda}_n^{(j)}+\tau\delta_j|^{p-2}}
{|z+m_n^{(j)}(z)|^2|z+m_n(z)+s(z)|}.\notag
\end{align}
Furthermore, we use inequality \eqref{lar1}  and Lemma \ref{exp0} in the Appendix. We get, for $z\in\mathbb G$,
\begin{align}
 |G_1|&\le \frac{Cp}{n^2}\sum_{j=1}^n\E\frac{|\varepsilon_{j\nu}||1+\eta_{j0}||\Lambda_n^{(j)}||\widetilde{\Lambda}_n^{(j)}|^{p-2}}
{|z+m_n^{(j)}(z)|^2|z+m_n^{(j)}(z)+s(z)|}\notag\\&\qquad\q+\frac{(Cp)^{p-2}}{n^2}\sum_{j=1}^n\E\frac{|\varepsilon_{j\nu}(1+\eta_{j0})||\Lambda_n^{(j)}||\delta_j|^{p-2}}
{|z+m_n^{(j)}(z)|^2|z+m_n^{(j)}(z)+s(z)|}.\notag
\end{align}
Note that
\begin{equation}\label{new11}
 |{\widetilde\Lambda}_n^{(j)}|^p\le {\rm e}|\Lambda_n^{(j)}|^p+\frac{C^pp^p}{n^p}(1+|\eta_{j0}|)^p,
\end{equation}
This inequality and Lemma \ref{exp*} together imply
\begin{align}\label{g1}
 |G_1|&\le\frac{Cp}{n^2}\sum_{j=1}^n\E\frac{|\varepsilon_{j\nu}(1+\eta_{j0})||\Lambda_n^{(j)}|^{p-1}}{|z+m_n^{(j)}(z)|^2|z+m_n^{(j)}(z)+s(z)|}\notag\\&\qquad\q+\frac{Cp^{p-2}}{n^p}\sum_{j=1}^n
 \E\frac{|\varepsilon_{j\nu}||\Lambda_n^{(j)}|(1+|\eta_{j0}|)^{p-1}}{|z+m_n^{(j)}(z)|^2|z+m_n^{(j)}(z)+s(z)|}
 \notag\\&\qquad\q+\frac{(Cp)^{p-2}}{n^2}\sum_{j=1}^n\E\frac{|\varepsilon_{j\nu}(1+\eta_{j0})||\Lambda_n^{(j)}| |\delta_j|^{p-2}}
{|z+m_n^{(j)}(z)|^2|z+m_n^{(j)}(z)+s(z)|}.
\end{align}
Conditioning on $\mathfrak M^{(j)}$ and 
using Lemmas \ref{basic2}, \ref{basic5} and Lemma \ref{lam1}  in the Appendix, and inequality
\begin{equation}\label{222}
 |1+\eta_{j0}|\le v^{-1}\im(z+\im m_n^{(j)}(z))\le v^{-1}|z+m_n^{(j)}(z)+s(z)|,
\end{equation}
we obtain, for $z\in\mathbb G$ and for $\nu=2,3$ ,
\begin{align}
 |G_1|&\le\frac{Cp}{nv\sqrt{nv}}\frac1n\sum_{j=1}^n\E^{\frac{p-1}p}|\Lambda_n^{(j)}|^{p-1}+\frac{(Cp)^{p-2}}{(nv)^p}
 \notag\\&\qquad\q+\frac{(Cp)^{p-2}}{n^2v}\sum_{j=1}^n\E\frac{|\varepsilon_{j\nu}\Lambda_n^{(j)}| |\delta_j|^{p-2}}
{|z+m_n^{(j)}(z)|^2}.\notag
\end{align}

Without loss of generality we may assume that $p\ge 3$.
Applying Lemmas \ref{exp*} and \ref{exp0}, we get
\begin{align}\label{g10}
 |G_1|&\le\frac{Cp}{(nv)^{\frac32}}\E^{\frac{p-1}p}|\Lambda_n|^{p}+\frac{(Cp)^{p-2}}{(nv)^p}
 \notag\\&\qquad\q +\frac{(Cp)^{p-2}}{n^2v}\sum_{j=1}^n\E\frac{|\varepsilon_{j\nu}|\Lambda_n^{(j)}||\delta_j|^{p-2}}
{|z+m_n^{(j)}(z)|^2}.
\end{align}

According to definition \eqref{delta} 
\begin{equation}\label{delta1}
 |\delta_j|\le \frac1n|R_{jj}-s(z)|(1+|\eta_{j0}|)+\frac {|\eta_j||R_{jj}|}n.
\end{equation}
Notice that
\begin{equation}
 |R_{jj}(z)-s(z)|\le \frac{|R_{jj}|}{|z+m^{(j)}_n(z)|}(|\widehat\varepsilon_j|+|\Lambda_n^{(j)}|),
\end{equation}
where $\widehat\varepsilon_j=\varepsilon_{j1}+\varepsilon_{j2}+\varepsilon_{j3}$ and $\Lambda_n^{(j)}=\Lambda_n+\varepsilon_{j4}$.
Inequality \eqref{delta1} and equality \eqref{mn'} together imply
\begin{align}\label{delta2}
 |\delta_j|&\le \frac{|\Lambda_n^{(j)}||R_{jj}|}{nv}+\sum_{\mu=1}^3\frac{C|\varepsilon_{j\mu}|}{nv}|R_{jj}|+\frac {C|\eta_j||R_{jj}|}{n}.
\end{align}

Moreover,
\begin{equation}\label{636}
 |\Lambda_n^{(j)}|^q\le \text{\rm e}|\Lambda_n|^q+\frac{C^qq^q}{(nv)^q}.
\end{equation}
Inequalities \eqref{g10}, \eqref{mn'}, \eqref{new11} and Lemma \ref{exp*} 
yield, for $z\in\mathbb G$,
\begin{align}\label{6.37}
 |G_1|&\le \frac{Cp}{nv}\E^{\frac{p-1}p}|\Lambda_n|^{p}+\frac{(Cp)^{p-2}}{(nv)^p}\notag\\&\qquad \q+\frac{(Cp)^{p-2}}{n^2v}\sum_{j=1}^n\E\frac{|\varepsilon_{j\nu}||\Lambda_n^{(j)}|^{p-1}|R_{jj}|^{p-2}}
{|z+m_n^{(j)}(z)|^2(nv)^{p-2}}\notag\\&\qquad \q+\sum_{\mu=1}^3\frac C{n^2v}\sum_{j=1}^n\E\frac{(Cp)^{p-2}|\varepsilon_{j\nu}||\Lambda_n^{(j)}|
|\varepsilon_{j\mu}|^{p-2}}{(nv)^{p-2}|z+m_n^{(j)}(z)|^2}|R_{jj}|^{p-2}\notag\\&+
\frac {C^p}{n^2v}\sum_{j=1}^n\E\frac{|\varepsilon_{j\nu}||\Lambda_n^{(j)}||\eta_j|^{p-2}}{n^{p-2}|z+m_n^{(j)}(z)|^2}|R_{jj}|^{p-2}.
\end{align}
Applying H\"older's inequality, we get
\begin{align}\label{kkk}
 \E&\frac{|\varepsilon_{j\nu}||\Lambda_n^{(j)}|^{p-1}|R_{jj}|^{p-2}}
{|z+m_n^{(j)}(z)|^2}\notag\\&\le 
\E^{\frac14}\frac1{|z+m_n^{(j)}(z)|^{8}}
\E^{\frac14}|R_{jj}|^{4(p-2}
\E^{\frac14}|\Lambda_n^{(j)}|^{4(p-1)}
\E^{\frac1{4}}|\varepsilon_{j\nu}|^{4}.
\end{align}
Using that $|\Lambda_n^{(j)}|\le |\Lambda_n|+\frac1{nv}$, we arrive
\begin{align}
 \E&\frac{|\varepsilon_{j\nu}||\Lambda_n^{(j)}|^{p-1}|R_{jj}|^{p-2}}
{|z+m_n^{(j)}(z)|^2}\notag\\&\le 
\E^{\frac14}\frac1{|z+m_n^{(j)}(z)|^{8}}
\E^{\frac14}|R_{jj}|^{4(p-2}
(\E^{\frac14}|\Lambda_n|^{4(p-1)}+\frac{C^pp^{p-1}}{(nv)^{p-1}})
\E^{\frac1{4}}|\varepsilon_{j\nu}|^{4}
\end{align}
Note that
\begin{equation}\label{6.40}
 \E|\Lambda_n|^{4(p-1)}\le 
 \E|\frac1n\sum_{j=1}^n\varepsilon_jR_{jj}|^{2(p-1)}
 \le \frac1n\sum_{j=1}^n\E^{\frac12}|\varepsilon_j|^{4(p-1)}\E^{\frac12}|R_{jj}|^{4(p-1)}.
\end{equation}
Using condition \eqref{trun}, Lemmas \ref{bp1}, \ref{bp1*}, \ref{basic8}, Rosethal's  and Burkholder's inequalities and Corollary \ref{cor8} , we conclude that for $q\ge 2$
\begin{equation}\label{oo1}
 \E^{\frac14}{|\varepsilon_{j\mu}|^{4q}}\le \frac {(Cq)^q}{(nv)^{\frac q2}}
\end{equation}
and
\begin{equation}\label{oo2}
 \E^{\frac14}|\eta_j|^{4q}\le (Cq)^qn^{-\frac q2}v^{-\frac{3q}2}.
\end{equation}
The  inequalities \eqref{6.40},  \eqref{oo1} and Corollaries \ref{cor8}  together imply
\begin{align}
 \E\frac{|\varepsilon_{j\nu}||\Lambda_n^{(j)}|^{p-1}|R_{jj}|^{p-2}}
{|z+m_n^{(j)}(z)|^2}&\le \frac{ C^pp^{2p}}{(nv)^{\frac{p+1}4}}.
\end{align}
This implies that
\begin{align}
 \frac{(Cp)^{p-2}}{n^2v}\sum_{j=1}^n\E\frac{|\varepsilon_{j\nu}||\Lambda_n^{(j)}|^{p-1}|R_{jj}|^{p-2}}
{|z+m_n^{(j)}(z)|^2(nv)^{p-2}}\le \frac{ C^pp^{2p}}{(nv)^{\frac{5p-3}4}},
\end{align}
and, for $p\ge3$
\begin{align}
 \frac{(Cp)^{p-2}}{n^2v}\sum_{j=1}^n\E\frac{|\varepsilon_{j\nu}||\Lambda_n^{(j)}|^{p-1}|R_{jj}|^{p-2}}
{|z+m_n^{(j)}(z)|^2(nv)^{p-2}}\le \frac{ C^pp^{p}}{(nv)^p},
\end{align}
It is straightforward to check that this inequality holds for $p=2$ as well.
Similarly to \eqref{kkk}, applying H\"older's inequality and inequalities \eqref{oo1} and \eqref{oo2}, we get
\begin{align}\label{kk1}
\sum_{\mu=1}^3\frac C{n^2v}\sum_{j=1}^n\E\frac{(Cp)^{p-2}|\varepsilon_{j\nu}||\Lambda_n^{(j)}|
|\varepsilon_{j\mu}|^{p-2}}{(nv)^{p-2}|z+m_n^{(j)}(z)|^2}|R_{jj}|^{p-2} \le\frac{ C^pp^{p}}{(nv)^{p}}.
\end{align}
and
\begin{align}
\frac C{n^2v}\sum_{j=1}^n\E\frac{|\varepsilon_{j\nu}||\Lambda_n^{(j)}||\eta_j|^{p-2}}{n^{p-2}|z+m_n^{(j)}(z)|^2}|R_{jj}|^{p-2}\le\frac{ C^pp^{p}}{(nv)^{p}}
\end{align}
These inequalities and Corollary \ref{cor8} together imply
\begin{align}\label{g1*}
 |G_1|&\le \frac{Cp}{nv}J_p^{\frac{p-1}p}+\frac{(Cp)^{p}}{(nv)^p}.
\end{align}

To bound $G_{1+\mu}$, for $\mu=1,2,3$,   we use Lemma \ref{exp0} and inequality \eqref{222}. We get
\begin{align}
 |G_{1+\mu}|\le \frac{Cp}{n^2v}\sum_{j=1}^n\E\frac{|\varepsilon_{j\nu}||\varepsilon_{j\mu}||R_{jj}||\widetilde{\Lambda}_n^{(j)}+\tau\delta_j|^{p-2}}
{|z+m_n^{(j)}(z)|^2}.\notag
\end{align}
Furthermore, we use inequality \eqref{lar1}  and Lemma \ref{exp0} in the Appendix. We get, for $z\in\mathbb G$,
\begin{align}
 |G_{1+\mu}|&\le \frac{Cp}{n^2v}\sum_{j=1}^n\E\frac{|\varepsilon_{j\nu}||\varepsilon_{j\mu}|R_{jj}|||\widetilde{\Lambda}_n^{(j)}|^{p-2}}
{|z+m_n^{(j)}(z)|^2}\notag\\&\qquad\q+\frac{(Cp)^{p-2}}{n^2v}\sum_{j=1}^n\E\frac{|\varepsilon_{j\nu}||\varepsilon_{j\mu}||R_{jj}||\delta_j|^{p-2}}
{|z+m_n^{(j)}(z)|^2}.\notag
\end{align}

This inequality and Lemma \ref{exp*} together imply
\begin{align}\label{g1**}
 |G_{1+\mu}|&\le\frac{Cp}{n^2v}\sum_{j=1}^n\E\frac{|\varepsilon_{j\nu}||\varepsilon_{j\mu}||R_{jj}||\Lambda_n^{(j)}|^{p-2}}{|z+m_n^{(j)}(z)|^2}\notag\\&
 \qquad\q+\frac{Cp^{p-2}}{n^pv}\sum_{j=1}^n
 \E\frac{|\varepsilon_{j\nu}||\varepsilon_{j\mu}||R_{jj}|}{|z+m_n^{(j)}(z)|^2}
 \notag\\&\qquad\q+\frac{(Cp)^{p-2}}{n^2v}\sum_{j=1}^n\E\frac{|\varepsilon_{j\nu}||\varepsilon_{j\mu}||R_{jj}||\delta_j|^{p-2}}
{|z+m_n^{(j)}(z)|^2}.
\end{align}

To estimate the first sum on the right hand side of \eqref{g1**}, we  used  conditioning on $\mathfrak M^{(j)}$, then Corollaries \ref{corgot} and \ref{cor8} and 
Lemma \ref{exp*}. The estimations of the second and third sums on the right hand side of \eqref{g1**} are similar to \eqref{kk1}.
Similarly to \eqref{g1*} we get from the  last inequality , for $\mu=1,2,3$,
\begin{align}\label{g2}
 |G_{1+\mu}|&\le \frac{Cp}{nv}J_p^{\frac{p-1}p}+\frac{(Cp)^{p-2}}{(nv)^p}.
\end{align}
Combining inequalities \eqref{g1*} and \eqref{g2}, we get
\begin{align}\label{m31}
 |M_{31}|\le \frac{Cp}{nv}J_p^{\frac{p-1}p}+\frac{(Cp)^{p-2}}{(nv)^p}.
\end{align}

\subsubsection{Estimation of $M_{32}$}Recall that
\begin{align}\notag
 M_{32}&=\frac{1}{n^2}\sum_{j=1}^n\E\frac{\varepsilon_{j\nu}\eta_jR_{jj}\varphi'(\widetilde{\Lambda}_n^{(j)}+\tau\delta_j)}{(z+m_n(z)+s(z))(z+m_n^{(j)}(z))}.
\end{align}
Using inequality \eqref{lar1} and the definition of $\varphi$, we get
\begin{align}\notag
| M_{32}|\le \frac{p}{n^2}\sum_{j=1}^n\E\frac{|\varepsilon_{j\nu}\eta_j||R_{jj}||\widetilde{\Lambda}_n^{(j)}+\tau\delta_j|^{p-2}}{|z+m_n^{(j)}(z)+s(z)||z+m_n^{(j)}(z)|}.
\end{align}

Applying now Lemma \ref{exp*}, we obtain
\begin{align}
 |M_{32}|&\le \frac{p}{n^2}\sum_{j=1}^n\E\frac{|\varepsilon_{j\nu}||\eta_j||R_{jj}||\widetilde{\Lambda}_n^{(j)}|^{p-2}}{|z+m_n^{(j)}(z)+s(z)||z+m_n^{(j)}(z)|}
 \notag\\&+
 \frac{(Cp)^{p-1}}{n^2}\sum_{j=1}^n\E\frac{|\varepsilon_{j\nu}||\eta_j||R_{jj}||\delta_j|^{p-2}}{|z+m_n^{(j)}(z)+s(z)||z+m_n^{(j)}(z)|}.\notag
\end{align}
By Lemmas \ref{basic1}, \ref{basic2}, \ref{basic5}, we have
\begin{align}
 \E\{|\varepsilon_{j\nu}||\eta_j|\Big|\mathfrak M^{(j)}\}&\le \frac{C\im m_n^{(j)}}{(nv^2)}.\notag
\end{align}
We get, for $z\in\mathbb G$,
\begin{align}\label{m32}
 |M_{32}|&\le\frac{Cp}{(nv)^2}\frac1n\sum_{j=1}^n\E^{\frac{p-2}p}|\widetilde{\Lambda}_n^{(j)}|^{p}\notag\\&\qquad\q+
 \frac{(Cp)^{p-1}}{n^2}\sum_{j=1}^n\E\frac{|\varepsilon_{j\nu}||\eta_j||\delta_j|^{p-2}}{|z+m_n^{(j)}(z)+s(z)||z+m_n^{(j)}(z)|}.
\end{align}
Inequalities \eqref{m32} and \eqref{delta2} yield
\begin{align}\label{rr1}
 |M_{32}|&\le\frac{Cp}{(nv)^2}\frac1n\sum_{j=1}^n\E^{\frac{p-2}p}|\widetilde{\Lambda}_n^{(j)}|^{p}\notag\\&\qquad\q+
 \frac{(Cp)^{p-1}}{n^p}\sum_{j=1}^n\E\frac{|\varepsilon_{j\nu}||\eta_j|| |\Lambda_n|^{p-2}|R_{jj}|^{p-2}}{|z+m_n^{(j)}(z)+s(z)||z+m_n^{(j)}(z)|}\notag\\&
 \qquad\q+
\sum_{q=1}^4 \frac{(Cp)^{p-1}}{n^p}\sum_{j=1}^n\E\frac{|\varepsilon_{j\nu}||\varepsilon_{j\mu}||\varepsilon_{jq}|^{p-2}(1+|\eta_j|)^{p-2}|R_{jj}|^{p-2}}
{|z+m_n^{(j)}(z)+s(z)||z+m_n^{(j)}(z)||z+m_n(z)|^{p-2}}.
\end{align}

Applying Lemma \ref{exp*} to estimate the first sum on the right hand side of \eqref{rr1}, Corollary \ref{corgot} to estimate the second one and 
inequalities \eqref{oo1} and \eqref{oo2}, we get
\begin{equation}\label{mm100}
 |M_{32}|\le\frac{Cp}{(nv)}\E^{\frac{p-1}p}|\Lambda_n|^p+
 \frac{(Cp)^{p-1}}{(nv)^{p}} .
\end{equation}
The representation \eqref{m3} and inequalities \eqref{m31} and \eqref{mm100} together imply
\begin{equation}\notag
 |M_{3}|\le\frac{Cp}{(nv)}J_p^{\frac{p-1}p}+
 \frac{(Cp)^{p-1}}{(nv)^{p}}.
\end{equation}

\subsection{Estimation of $M_4$} Recall that
 \begin{equation}\notag
  M_4=-\frac1n\sum_{j=1}^n\E\frac{\varepsilon_{j\nu}\frac1{z+m_n^{(j)}(z)}\varphi(\widetilde{\Lambda}_n^{(j)})\varepsilon_{j4}}{(z+m_n^{(j)}(z)+s(z))(z+m_n(z)+s(z))}.
 \end{equation}
Consider the following moments
\begin{align}\notag
 \gamma_n:=\E\{\varepsilon_{j\nu}{\varepsilon_{j4}}\Big|\mathfrak M^{(j)}\}.
\end{align}
By Cauchy -- Schwartz inequality we have
\begin{equation}\notag
 |\gamma_n|\le \E^{\frac12}\{|\varepsilon_{j\nu}|^2\Big|\mathfrak M^{(j)}\}\E^{\frac12}\{|\varepsilon_{j4}|^2\Big|\mathfrak M^{(j)}\}.
\end{equation}
By Lemmas \ref{basic2} and \ref{basic5}, we obtain
\begin{equation}\label{jnu}
\E^{\frac12}\{|\varepsilon_{j\nu}|^2\Big|\mathfrak M^{(j)}\}\le C(nv)^{-\frac12}\im^{\frac12}m_n^{(j)}(z).
\end{equation}

Furthermore,
\begin{equation}\notag
\E\{|\varepsilon_{j4}|^2\Big|\mathfrak M^{(j)}\}\le \frac C{n^2}\E\{(1+|\eta_{j0}|^2+|\eta_j|^2)|R_{jj}|^2\Big|\mathfrak M^{(j)}\}.
\end{equation}

It is straightforward to check that
\begin{align}\label{j4}
 \E\{|\varepsilon_{j4}|^2\Big|\Big|\mathfrak M^{(j)}\}\le (\frac{1+v^{-2}\im^2m_n^{(j)}(z)}{n^2}+\frac C{n^3v^3}\im m_n^{(j)}(z))\E^{\frac12}\{|R_{jj}|^4\Big|\mathfrak M^{(j)}\}.
\end{align}
The inequalities \eqref{jnu} and \eqref{j4} together imply
\begin{equation}\label{gamma}
 |\gamma_n|\le (\frac {C\im^{\frac12}m_n^{(j)}(z)}{n\sqrt{nv}}+\frac {C\im m_n^{(j)}(z)}{(nv)^2}+\frac {C\im ^{\frac32}m_n^{(j)}(z)}{(nv)^{\frac32}})\E^{\frac12}\{|R_{jj}|^4|\mathfrak M^{(j)}\}.
\end{equation}

Using inequality \eqref{lar1} and conditioning on $\mathfrak M^{(j)}$, we may write, for $z\in\mathbb G$
\begin{align}
 |M_4|&\le \frac Cn\sum_{j=1}^n\E|\gamma_n|\frac{\frac1{|z+m_n^{(j)}(z)|}|\varphi(\widetilde{\Lambda}_n^{(j)})|}
 {|z+m_n^{(j)}(z)+s(z)|^2}.\notag
\end{align}
Applying now inequality \eqref{gamma}, we get
\begin{align}
 |M_4|&\le \frac1{n\sqrt{nv}|z^2-4|^{\frac34}}\frac Cn\sum_{j=1}^n\E\frac1{|z+m_n^{(j)}(z)|}|\varphi(\widetilde{\Lambda}_n^{(j)})|\E^{\frac12}\{|R_{jj}|^4\Big|\mathfrak M^{(j)}\}\notag\\&\qquad\q+
 \frac1{(nv)^2|z^2-4|^{\frac12}}\frac Cn\sum_{j=1}^n\E\frac1{|z+m_n^{(j)}(z)|}|\varphi(\widetilde{\Lambda}_n^{(j)})|\E^{\frac12}\{|R_{jj}|^4\Big|\mathfrak M^{(j)}\}\notag\\&
 +\frac1{(nv)^{\frac32}|z^2-4|^{\frac14}}\frac Cn\sum_{j=1}^n\E\frac1{|z+m_n^{(j)}(z)|}|\varphi(\widetilde{\Lambda}_n^{(j)})|\E^{\frac12}\{|R_{jj}|^4\Big|\mathfrak M^{(j)}\}.
 \end{align}
Applying H\"older's inequality and Corollary \ref{cor8}, we obtain
\begin{align}\label{m4}
 |M_4|&\le \frac {Cp}{nv}J_p^{\frac{p-1}p}+\frac {(Cp)^{p-1}}{(nv)^p}.
\end{align}
Combining now inequalities \eqref{lada3}, \eqref{m1}, \eqref{lada3}, \eqref{m3}, \eqref{m4}, we get
\begin{align}\notag
 |\mathfrak T_2|\le \frac {Cp}{nv}J_p^{\frac{p-1}p}+\frac {(Cp)^{p-1}}{(nv)^p}.
\end{align}
Together with \eqref{fini1} we get 
\begin{equation}
 |J_p|\le \frac {Cp}{nv}J_p^{\frac{p-1}p}+\frac {(Cp)^{p-1}}{(nv)^p}+\frac {Cp}{nv}J_p.
\end{equation}
Since $\frac {Cp}{nv}<c<1$, we conclude
\begin{equation}\notag
 |J_p|\le \frac {Cp}{nv}J_p^{\frac{p-1}p}+\frac {(Cp)^{p-1}}{(nv)^p}.
\end{equation}
Using Lemma  \ref{frak} in the Appendix, we get, for $z\in\mathbb G$,
\begin{equation}\notag
 |J_p|\le\frac {(Cp)^{p-1}}{(nv)^p}.
\end{equation}
Thus Theorem \ref{stiltjesmain} is proved.

\section{Appendix}
\subsection{Rosenthal's and Burkholder's Inequalities}
In this subsection we state the Rosenthal and Burkholder inequalities starting with Rosenthal's inequality.
Let $\xi_1,\ldots,\xi_n$ be independent random variables with $\E\xi_j=0$, $\E\xi_j^2=1$ and for $p\ge 1$ $\E|\xi_j|^p\le \mu_p$ for $j=1,\ldots,n$.
\begin{lem}\label{Rosent}{\rm (Rosenthal's inequality)}

 There exists an absolute constant $C_1$ such that
 \begin{equation}\notag
 \E|\sum_{j=1}^na_j\xi_j|^p\le C_1^pp^p\Big(\big(\sum_{j=1}^p|a_j|^2\big)^{\frac p2}+\mu_p\sum_{j=1}^p|a_j|^p\Big)
 \end{equation}
\end{lem}
\begin{proof}
 For the proof of this inequality see \cite{Rosenthal:1970} and \cite{Johnson:1985}.
\end{proof}
Let $\xi_1,\ldots\xi_n$ be martingale-difference with respect to $\sigma$-algebras $\mathfrak M_j=\sigma(\xi_1,\ldots,\xi_{j-1})$.
Assume that  $\E\xi_j^2=1$ and $\E|\xi_j|^p<\infty$.
\begin{lem}\label{burkh}{\rm (Burkholder's inequality)}
 There exist an absolute constant $C_2$ such that
 \begin{equation}\notag
  \E|\sum_{j=1}^n\xi_j|^p\le C_2^pp^p\Big(\Big(\E(\sum_{k=1}^n\E\{\xi_k^2|\mathfrak M_{k-1}\}\Big)^{\frac p2}+\sum_{k=1}^p\E|\xi_k|^p\Big).
 \end{equation}

\end{lem}
\begin{proof}
 For the proof of this inequality see \cite{Burkholder:1973} and \cite{Hitczenko:1990}.
\end{proof}
We rewrite the Burkholder inequality for quadratic forms in independent random variables.
Let $\zeta_1,\ldots,\zeta_n$ be independent random variables such that $\E\zeta_j=0$, $\E|\eta_j|^2=1$ and $\E|\zeta_j|^p\le \mu_p$. 
Let $a_{ij}=a_{ji}$ for all $i,j=1,\ldots n$.
Consider the quadratic form
\begin{equation}\notag
 Q=\sum_{1\le j\ne k\le n}a_{jk}\zeta_j\zeta_k.
\end{equation}
\begin{lem}\label{burkh1}
 There exists an absolute constant $C_2$ such that
 \begin{equation}
  \E|Q|^p\le C_2^p\Big(\E\big(\sum_{j=2}^{n}(\sum_{k=1}^{j-1}a_{jk}\zeta_k)^2\big)^{\frac p2}+\mu_p\sum_{j=2}^n\E|\sum_{k=1}^{j-1}a_{jk}\zeta_k|^p\Big).
 \end{equation}

\end{lem}
\begin{proof}
 Introduce the random variables 
 \begin{equation}\notag
  \xi_j=\zeta_j\sum_{k=1}^{j-1}a_{jk}\zeta_k, \q j=2,\ldots,n.
 \end{equation}
 It is straightforward to check that
 \begin{equation}\notag
  \E\{\xi_j|\mathfrak M_{j-1}\}=0,
 \end{equation}
 and that $\xi_j$ are $\mathfrak M_{j}$ measurable.
 Hence $\xi_1,\ldots,\xi_n$ are martingale-differences.
 We may write
 \begin{equation}\notag
  Q=2\sum_{j=2}^n\xi_j
 \end{equation}
Applying now Lemma \ref{burkh} and using
\begin{align}
 \E\{|\xi_j|^2|\mathfrak M_{j-1}\}&=(\sum_{k=1}^{j-1}a_{jk}\eta_k)^2\E\zeta_j^2,\notag\\
 \E|\xi_j|^p&=\E|\eta_j|^p\E|\sum_{k=1}^{j-1}a_{jk}\zeta_j|^p,\notag
\end{align}
we get the claim.
Thus, Lemma \ref{burkh1} is proved.

\end{proof}
\begin{lem}\label{sum1}
 Assuming the  conditions of Theorem \ref{cormain} there exists a positive  constant $C=C(\mu_4, D)$, depending on $\mu_4$ and $D$ such that, for any $1\le q\le C\log n$,
 \begin{equation}\notag
  \E(\frac1n\sum_{j=1}^nX_{jj}^2)^q\le C^q.
 \end{equation}

\end{lem}
\begin{proof}
 Applying the triangle inequality, we get
 \begin{equation}\notag
  \E(\frac1n\sum_{j=1}^nX_{jj}^2)^q\le 2^q(1+\frac1{n^q}\E|\sum_{j=1}^n(X_{jj}^2-1)|^q).
 \end{equation}
Using now Rosenthal's inequality, we get
\begin{equation}\notag
 \E(\frac1n\sum_{j=1}^nX_{jj}^2)^q\le 2^q(1+\frac1{n^q}(C_1^qq^qn^{\frac q2}+n\max_{jj}\E|X_{jj}|^{2q}).
\end{equation}
According to condition \eqref{trun}, we have
\begin{equation}\notag
 \E(\frac1n\sum_{j=1}^nX_{jj}^2)^q\le 2^q(1+(C_1^qq^qn^{-\frac q2}+D^{2q-4}n^{-\frac q2}\mu_4).
\end{equation}

\end{proof}
\begin{cor}\label{eps1}
 Under the condition of Theorem \ref{cormain} there exists a positive  constant $C=C(\mu_4, D)$, depending on $\mu_4$ and $D$ such that, for any $1\le q\le C\log n$,
 \begin{equation}\notag
  \E(\frac1n\sum_{j=1}^n|\varepsilon_{j1}|^2)^q\le \frac{C^q}{n^q}.
 \end{equation}

\end{cor}
\begin{proof}
 The result immediately follows from the definition 
 \begin{equation}\notag
  \varepsilon_{j1}=\frac1{\sqrt n}X_{jj},
 \end{equation}
and Lemma \ref{sum1}.
\end{proof}

The next Lemma describes the behavior of the moments of  $\varphi(\Lambda_n)$.
Recall that 
\begin{equation}\notag
\Lambda_n=m_n(z)-s(z),\quad \Lambda_n^{(j)}=m_n^{(j)}-s(z),\quad  \varepsilon_{j4}=\Lambda_n-\Lambda_n^{(j)},
\end{equation}
and
\begin{equation}\notag
 \varphi(z)=\overline z|z|^{p-2}.
\end{equation}
First we prove
\begin{lem}\label{frak}
 Let  $t>r\ge 1$  and $a,b>0$. Any $x>0$ satisfying the inequality 
  \begin{equation}\label{w1*}
  x^t\le a+bx^r
 \end{equation}
is explicitly bounded as follows
 \begin{equation}\label{w2*}
  x^t\le \text{\rm e}a+\left(\frac{2t-r}{t-r}\right)^{\frac t{t-r}}b^{\frac t{t-r}}.
 \end{equation}

\end{lem}
\begin{proof}
 First assume that $x\le a^{\frac1t}$. Then inequality \eqref{w2*} holds. If $x\ge a^{\frac1t}$, then according to inequality \eqref{w1*}
 \begin{equation}
  x^{t-r}\le a^{\frac{t-r}t}+b,
 \end{equation}
or
\begin{equation}
 x^t\le (a^{\frac{t-r}t}+b)^{\frac{t}{t-r}}.
\end{equation}
Using that for any $\alpha>0$ and $a>0,b>0$ 
\begin{equation}
 (a+b)^{\alpha}\le(a+\frac a{\alpha})^{\alpha}+(b+ {\alpha} b)^{\alpha} \le \text{e}a^{\alpha}+(1+\alpha)^{\alpha}b^{\alpha},
\end{equation}
we get the claim.
\end{proof}

\begin{cor}\label{simple}
 Assume that for $a,b,c,x>0$ the following inequality holds
 \begin{equation}\notag
  x^t\le a+bx^{t-1}+cx^{t-2}.
 \end{equation}
Then
\begin{equation}\notag
 x^t\le \text{\rm e}^2a+\text{\rm e}\left(1+\frac t2\right)^{\frac t2}c^{\frac p2}+t^{t}\text{\rm e}^{t}b^{t}.
\end{equation}

\end{cor}
\begin{proof}
 We apply Lemma \ref{frak} with $a'=a+bx^{p-1}$, $b'=c$ and $r=t-2$ and obtain
 \begin{equation}\notag
  x^t\le \text{\rm e}a+\text{\rm e}bx^{t-1}+\left(1+\frac t2\right)^{\frac t2}c^{\frac t2}.
 \end{equation}
Using Lemma \ref{frak} again with $a''=\text{\rm e}a+\left(1+\frac t2\right)^{\frac t2}c^{\frac t2}$, $b''=\text{\rm e}b$ and $r=t-1$, we get
\begin{equation}\notag
 x^t\le \text{\rm e}^2a+\text{\rm e}\left(1+\frac t2\right)^{\frac t2}c^{\frac t2}+t^{t}\text{\rm e}^{t}b^{t}.
\end{equation}

\end{proof}

\begin{lem}\label{exp0} Recall that $\varepsilon_{j4}=\Lambda_n-\Lambda_n^{(j)}$. Then
\begin{equation}\notag
 |\varphi(\Lambda_n)-\varphi(\Lambda_n^{(j)})|\le p|\varepsilon_{j4}|\E_{\tau}|\Lambda_n-\tau\varepsilon_{j4}|^{p-2},
\end{equation}
where $\tau$ denotes a random variable which is uniformly distributed on $[0,1]$  and independent of all $X_{jk}$, for $j,k=1,\ldots,n$.

\end{lem}
\begin{proof}
For $x\in[0,1]$ define  the function, 
 \begin{equation}\notag
 \widehat\varphi(x)=\varphi(\Lambda_n-x\varepsilon_{j4}).
 \end{equation}
 It is easy to see that $\widehat \varphi(0)=\varphi(\Lambda_n)$, $\widehat \varphi(1)=\varphi(\Lambda_n^{(j)})$.
 By Taylor's formula we have
 \begin{equation}\notag
  \varphi(\Lambda_n)-\varphi(\Lambda_n^{(j)})=-\varepsilon_{j4}\E_{\tau}\widehat\varphi'(\Lambda-\tau\varepsilon_{j4}).
 \end{equation}
It is straightforward to check that
\begin{equation}\notag
 |\widehat\varphi'(x)|\le p|\Lambda_n-x\varepsilon_{j4}|^{p-2}.
\end{equation}

\end{proof}
\begin{lem}\label{exp*}
With the notations of Lemma \ref{exp0} we have for any $q\ge1$ and for all $\zeta\in\mathbb C$
 \begin{equation}\notag
  |\Lambda_n-\tau\zeta|^{q}\le {(q+1)^q}|\zeta|^q+{\rm e}|\Lambda_n|^q.
 \end{equation}

\end{lem}

\begin{proof}
We observe that
\begin{equation}\notag
|\Lambda_n-\tau\zeta|^{q}\le  |\Lambda_n-\tau\zeta|^{q}\mathbb I\{|\Lambda|\le q|\zeta|\}+
 |\Lambda_n-\tau\varepsilon_{j4}|^{q}\mathbf I\{|\Lambda_n|\le  q|\zeta|\}.
\end{equation}
From here  we conclude
\begin{equation}\notag
|\Lambda_n-\tau\zeta|^{q}\le {(q+1)^q}|\zeta|^q+
(1+\frac 1q)^q|\Lambda_n|^{q}\le {(q+1)^q}|\zeta|^q+
{\rm e} |\Lambda_n|^{q}.
\end{equation}
 Thus Lemma \ref{exp*} is proved.
\end{proof}

\subsection{Auxiliary Inequalities for Resolvent Matrices} We shall use the following relation between resolvent matrices. Let $\mathbb A$ and $\mathbb B$ be two Hermitian matrices and 
let $\mathbf R_{\mathbf A}=(\mathbb A-z\mathbf I)^{-1}$ and $\mathbf R_{\mathbf B}=(\mathbb B-z\mathbf I)^{-1}$ denote their resolvent matrices.
Recall the resolvent equality
\begin{equation}\label{reseq}
 \mathbf R_{\mathbf A}-\mathbf R_{\mathbf B}=\mathbf R_{\mathbf A}(\mathbf B-\mathbf A)\mathbf R_{\mathbf B}=
 -\mathbf R_{\mathbf B}(\mathbf B-\mathbf A)\mathbf R_{\mathbf A}.
\end{equation}
Recall the equation, for $j\in\mathbb T_{\mathbb J}$, and $\mathbb J\subset\mathbb T$ (compare with \eqref{repr001})
\begin{equation}
 R_{jj}^{(\mathbb J)}=-\frac1{z+m_n^{(\mathbb J)}(z)}+\frac1{z+m_n^{(\mathbb J)}(z)}\varepsilon_j^{(\mathbb J)}R^{(\mathbb J)}_{jj},
\end{equation}
where 
\begin{align}
\varepsilon_{j1}^{(\mathbb J)}&=\frac{X_{jj}}{\sqrt n},\quad
\varepsilon_{j2}^{(\mathbb J)}=\frac1n\sum_{l\ne k\in\mathbb T_{\mathbb J,j}}X_{jl}X_{jk}R^{(\mathbb J,j)}_{kl},\notag\\
\varepsilon_{j3}^{(\mathbb J)}&=\frac1n\sum_{l\in\mathbb T_{\mathbb J,j}}(X_{jl}^2-1)R^{(\mathbb J,j)}_{ll},\quad
\varepsilon_{j4}^{(\mathbb J)}=m_n^{(\mathbb J)}(z)-m_n^{(\mathbb J,j)}(z).
\end{align}
Summing these equations for $j\in\mathbb T_{\mathbb J}$, we get
\begin{equation}\label{main1}
m_n^{(\mathbb J)}(z)=-\frac{n-|\mathbb J|}{n(z+m_n^{(\mathbb J))}(z)}+\frac{T_n^{(\mathbb J)}}{z+m_n^{(\mathbb J)}(z)},
\end{equation}

where 
\begin{equation}
 T_n^{(\mathbb J)}=\frac1n\sum_{j=1}^n\varepsilon_j^{(\mathbb J)}R_{jj}^{(\mathbb J)}.
\end{equation}
Note that
\begin{equation}\label{main2}
 \frac1{z+m_n^{(\mathbb J)}(z)}=\frac1{z+s(z)}-\frac{m_n^{(\mathbb J)}(z)-s(z)}{(s(z)+z)(z+m_n^{(\mathbb J)}(z))}=
 -s(z)+\frac{s(z)\Lambda_n^{(\mathbb J)}(z)}{z+m_n^{(\mathbb J)}(z)},
\end{equation}
where 
\begin{equation}
 \Lambda_n^{(\mathbb J)}=\Lambda_n^{(\mathbb J)}(z)=m_n^{(\mathbb J)}(z)-s(z).
\end{equation}
Equalities \eqref{main1} and \eqref{main2} together imply 
\begin{equation}
 \Lambda_n^{(\mathbb J)}=-\frac{s(z)\Lambda_n^{(\mathbb J)}}{z+m_n^{(\mathbb J)}(z)}+\frac{T_n^{(\mathbb J)}}{z+m_n^{(\mathbb J)}(z)}+\frac{|\mathbb J|}{n(z+m_n^{(\mathbb J)}(z))}.
\end{equation}
Solving this with respect to $\Lambda_n^{(\mathbb J)}$, we get
\begin{equation}\label{main3}
 \Lambda_n^{(\mathbb J)}=\frac{T_n^{(\mathbb J)}}{z+m_n^{(\mathbb J)}(z)+s(z)}+\frac{|\mathbb J|}{n(z+m_n^{(\mathbb J)}(z)+s(z))}.
\end{equation}

\begin{lem}\label{resol00}
For any $z=u+iv$ with $v>0$ and for any $\mathbb J\subset \mathbb T$, we have
\begin{align}\label{res1}
 \frac1n\sum_{l,k\in\mathbb T_{\mathbb J}}|R^{(\mathbb J)}_{kl}|^2\le v^{-1}\im m_n^{(\mathbb J)}(z).
\end{align}
For any  $l\in\mathbb T_{\mathbb J}$
\begin{equation}\label{res2}
 \sum_{k\in\mathbb T_{\mathbb J}}|R^{(\mathbb J)}_{kl}|^2\le v^{-1}\im R^{(\mathbb J)}_{ll}.
\end{equation}
and
\begin{equation}\label{res20}
 \sum_{k\in\mathbb T_{\mathbb J}}|[(R^{(\mathbb J)})^2]_{kl}|^2\le v^{-3}\im R^{(\mathbb J)}_{ll}.
\end{equation}
Moreover, for any $\mathbb J\subset T$ and for any $l\in\mathbb T_{\mathbb J}$ we have
\begin{align}\label{res3}
\frac1n\sum_{l\in\mathbb T_{\mathbb J}}|[(R^{(\mathbb J)})^2]_{ll}|^2\le v^{-3}\im m_n^{(\mathbb J)}(z),
\end{align}
and, for any $p\ge1$
\begin{align}\label{res4}
\frac1n\sum_{l\in\mathbb T_{\mathbb J}}|[(R^{(\mathbb J)})^2]_{ll}|^p\le v^{-p}\frac1n\sum_{l\in\mathbb T_{\mathbb J}}\im^pR^{(\mathbb J)}_{ll}.
\end{align}
Finally,
\begin{align}\label{res5}
\frac1n\sum_{l,k\in\mathbb T_{\mathbb J}}|[(R^{(\mathbb J)})^2]_{lk}|^2\le v^{-3}\im m_n^{(\mathbb J)}(z),
\end{align}
and
\begin{align}\label{res6}
\frac1n\sum_{l,k\in\mathbb T_{\mathbb J}}|[(R^{(\mathbb J)})^2]_{lk}|^{2p}\le v^{-3p}\frac1n\sum_{l\in\mathbb T_{\mathbb J}}\im^pR^{(\mathbb J)}_{ll},
\end{align}
We have as well
\begin{align}\label{res7}
\frac1{n^2}\sum_{l,k\in\mathbb T_{\mathbb J}}|[(R^{(\mathbb J)})^2]_{lk}|^{2p}\le v^{-2p}(\frac1n\sum_{l\in\mathbb T_{\mathbb J}}\im^pR^{(\mathbb J)}_{ll})^2.
\end{align}
\end{lem}
\begin{proof}For $l\in\mathbb T_{\mathbb J}$ let us denote by $\lambda^{(\mathbb J)}_l$ for $l\in\mathbb T_{\mathbb J}$ the eigenvalues of the matrix $\mathbf W^{(\mathbb J)}$.
Then we may write (compare \eqref{orth1}) 
\begin{equation}
\frac1n\sum_{l,k\in\mathbb T_{\mathbb J}}|R^{(\mathbb J)}_{kl}|^2\le\frac1n\sum_{l\in\mathbb T_{\mathbb J}}\frac1{|\lambda^{(\mathbb J)}_l-z|^2}.
 \end{equation}
 Note that, for any $x\in\mathbb R^1$
 \begin{equation}
  \im\frac1{x-z}=\frac{v}{|x-z|^2}.
 \end{equation}
We may write
\begin{equation}
 \frac1{|\lambda^{(\mathbb J)}_l-z|^2}=v^{-1}\im\frac1{\lambda^{(\mathbb J)}_l-z}
\end{equation}
and
\begin{equation}
\frac1n\sum_{l,k\in\mathbb T_{\mathbb J}}|R^{(\mathbb J)}_{kl}|^2\le v^{-1}\im(\frac1n\sum_{l\in\mathbb T_{\mathbb J}}\frac1{\lambda^{(\mathbb J)}_l-z})
=v^{-1}\im m_n^{(\mathbb J)}(z).
 \end{equation}
Thus inequality \eqref{res1} is proved.
 Let denote now by $\mathbf u_l^{(\mathbb J)}=(u^{(\mathbb J)}_{lk})_{k\in\mathbb T_{\mathbb J}}$ the eigenvector of the matrix $\mathbf W^{(\mathbb J)}$ 
 corresponding to the eigenvalue $\lambda^{(\mathbb J)}_l$.
Using  this notation we may write
 \begin{equation}\label{orth1}
  R^{(\mathbb J)}_{lk}=\sum_{q\in\mathbb T_{\mathbb J}}\frac1{\lambda_q^{(\mathbb J)}-z}u^{(\mathbb J)}_{lq}u^{(\mathbb J)}_{kq}.
 \end{equation}
It is straightforward to check that the following inequality holds
\begin{align}
 \sum_{k\in\mathbb T_{\mathbb J}}|R^{(\mathbb J)}_{kl}|^2&\le\sum_{q\in\mathbb T_{\mathbb J}}\frac1{|\lambda^{(\mathbb J)}_q-z|^2}|u^{(\mathbb J)}_{lq}|^2\notag\\&=
 v^{-1}\im\Big(\sum_{q\in\mathbb T_{\mathbb J}}\frac1{\lambda^{(\mathbb J)}_q-z}|u^{(\mathbb J)}_{lq}|^2\Big)=v^{-1}\im R_{ll}^{(\mathbb J)}.
\end{align}
Thus, inequality \eqref{res2} is proved.
Similarly we get
\begin{equation}
 \sum_{k\in\mathbb T_{\mathbb J}}|[(R^{(\mathbb J)})^2]_{kl}|^2\le\sum_{q\in\mathbb T_{\mathbb J}}\frac1{|\lambda^{(\mathbb J)}_q-z|^4}|u^{(\mathbb J)}_{lq}|^2\le
 v^{-3}\im R^{(\mathbb J)}_{ll}.
\end{equation}
This proves  inequality \eqref{res20}. 
To prove  inequality \eqref{res3} we observe that
\begin{equation}\label{resol1} 
 |[(R^{(\mathbb J)})^2]_{ll}|\le \sum_{k\in\mathbb T_{\mathbb J}}|R^{(\mathbb J)}_{lk}|^2.
\end{equation}
This inequality implies
\begin{equation}
 \frac1n\sum_{l\in\mathbb T_{\mathbb J}}|[(R^{(\mathbb J)})^2]_{ll}|^2\le \frac1n\sum_{l\in\mathbb T_{\mathbb J}}
 (\sum_{k\in\mathbb T_{\mathbb J}}|R^{(\mathbb J)}_{lk}|^2)^2.
\end{equation}
Applying now inequality \eqref{res2}, we get
\begin{align}
\frac1n\sum_{l\in\mathbb T_{\mathbb J}}|[(R^{(\mathbb J)})^2]_{ll}|^2\le v^{-2}\frac1n\sum_{l\in\mathbb T_{\mathbb J}}\im^2R^{(\mathbb J)}_{ll}.
\end{align}
Using  $|R^{(\mathbb J)}_{ll}|\le v^{-1}$ this leads to the following bound
\begin{align}
\frac1n\sum_{l\in\mathbb T_{\mathbb J}}|[(R^{(\mathbb J)})^2]_{ll}|^2\le v^{-3}\frac1n\sum_{l\in\mathbb T_{\mathbb J}}\im R^{(\mathbb J)}_{ll}=
v^{-3}\im m_n^{(\mathbb J)}(z).
\end{align}
Thus inequality \eqref{res3} is proved.
Furthermore, applying inequality \eqref{resol1}, we may write
\begin{align}
\frac1n\sum_{l\in\mathbb T_{\mathbb J}}|[(R^{(\mathbb J)})^2]_{ll}|^4\le\frac1n\sum_{l\in\mathbb T_{\mathbb J}}
 (\sum_{k\in\mathbb T_{\mathbb J}}|R^{(\mathbb J)}_{lk}|^2)^4.
\end{align}
Applying \eqref{res2}, this inequality yields
\begin{align}
\frac1n\sum_{l\in\mathbb T_{\mathbb J}}|[(R^{(\mathbb J)})^2]_{ll}|^4\le v^{-4}\frac1n\sum_{l\in\mathbb T_{\mathbb J}}
 \im^4R^{(\mathbb J)}_{ll}.
\end{align}
The last inequality proves  inequality \eqref{res4}.
Note that
\begin{align}
\frac1n\sum_{l,k\in\mathbb T_{\mathbb J}}|[(R^{(\mathbb J)})^2]_{lk}|^2&\le \frac1n\Tr|\mathbf R^{(\mathbb J)}|^4=\frac1n\sum_{l\in\mathbb T_{\mathbb J}}
\frac1{|\lambda^{(\mathbb J)}_l-z|^4}\notag\\&\le v^{-3}\im\frac1n\sum_{l\in\mathbb T_{\mathbb J}}\frac1{\lambda^{(\mathbb J)}_l-z}=v^{-3}\im m_n^{(\mathbb J)}(z).
\end{align}
Thus, inequality \eqref{res5} is proved.
To finish we note that
\begin{align}
 \frac1n\sum_{l,k\in\mathbb T_{\mathbb J}}|[(R^{(\mathbb J)})^2]_{lk}|^4\le \frac1n\sum_{l\in\mathbb T_{\mathbb J}}
 (\sum_{k\in\mathbb T_{\mathbb J}}|[(R^{(\mathbb J)})^2]_{lk}|^2)^2.
\end{align}
Applying inequality \eqref{res20}, we get
\begin{align}
 \frac1n\sum_{l,k\in\mathbb T_{\mathbb J}}|[(R^{(\mathbb J)})^2]_{lk}|^4\le v^{-6}\frac1n\sum_{l\in\mathbb T_{\mathbb J}}
 (\im R^{(\mathbb J)}_{ll})^2.
\end{align}
To prove inequality \eqref{res7}, we note 
\begin{equation}
|[(R^{(\mathbb J)})^2]_{lk}|^2\le(\sum_{q\in\mathbb T_{\mathbb J}}|R^{(\mathbb J)}_{lq}|^2)(\sum_{q\in\mathbb T_{\mathbb J}}|R^{(\mathbb J)}_{kq}|^2). 
\end{equation}
This inequality implies
\begin{align}
 \frac1{n^2}\sum_{l,k\in\mathbb T_{\mathbb J}}|[(R^{(\mathbb J)})^2]_{lk}|^{2p}\le 
 (\frac1n\sum_{l,k\in\mathbb T_{\mathbb J}}(\sum_{q\in\mathbb T_{\mathbb J}}|R^{(\mathbb J)}_{lq}|^2)^p)^2
 (\im R^{(\mathbb J)}_{ll})^2.
\end{align}
Applying inequality \eqref{res1}, we get the claim.
Thus, Lemma \ref{resol00} is proved.
\end{proof}

\begin{lem}\label{eps1*}
 Assuming the conditions of Theorem \ref{main}, we get
\begin{align}\notag
 \E|\varepsilon_{j1}|^2\le \frac {C}{n}.
\end{align}
\end{lem}
\begin{proof}
 The proof follows immediately from the definition of $\varepsilon_{j1}$ and the conditions of Theorem \ref{main}.
\end{proof}

\subsubsection{Some Auxiliary Bounds for Resolvent Matrices for $z=u+iV$ with $V=4$}
We shall use the bound for the $\varepsilon_{j\nu}$, and $\eta_{j}$ for $V=4$.
\begin{lem}\label{eps2}Assuming the conditions of Theorem \ref{main}, we get
\begin{align}\notag
 \E|\varepsilon_{j2}|^q\le \frac {C^qq}{n^{\frac q2}}.
\end{align}

\end{lem}
\begin{proof}
 Conditioning on $\mathfrak M^{(j)}$  and applying Burkholder's inequality \newline(see Lemma \ref{burkh1}), we get
 \begin{align}
  \E|\varepsilon_{j2}|^q\le C_2^qq^qn^{-q}(\E|\sum_{k\in\mathbb T_j}(\sum_{l=1}^{k-1}R^{(j)}_{kl}X_{jl})^2|^{\frac q2}+
  \mu_q\sum_{k\in\mathbb T_j}\E|\sum_{l=1}^{k-1}R^{(j)}_{kl}X_{jl}|^q).\notag
 \end{align}
Applying now Corollary \ref{q1*} and Rosenthal's inequality, we get
\begin{align}
 \E|\varepsilon_{j2}|^q\le C_2^qq^{2q}n^{-\frac q2}+\mu_qn^{-q}q^{2q}\sum_{l\in\mathbb T_j}(\sum_{k\in\mathbb T_j}|R^{(j)}_{kl}|^2)^{\frac q2}+
 \mu_q^2n^{-q}q^{2q}\sum_{k,l\in\mathbb T_j}|R^{(j)}_{kl}|^q.\notag
\end{align}
Using that $|R^{(j)}_{kl}|\le\frac14$ and $\sum_{l\in\mathbb T_j}|R^{(j)}_{kl}|^2\le \frac1{16}$ and $\mu_q\le D^{\frac q4}n^{-1}\mu_4$, we get
\begin{align}
 \E|\varepsilon_{j2}|^q\le C_2^qq^{2q}n^{-\frac q2}.\notag
\end{align}
Thus Lemma \ref{eps2} is proved.
\end{proof}
\begin{lem}\label{eps3}Assuming the conditions of Theorem \ref{main}, we get
\begin{align}
 \E|\varepsilon_{j3}|^q\le \frac {C^qq}{n^{\frac q2}}.\notag
\end{align}

\end{lem}

\begin{proof}
 Conditioning and applying  Rosenthal's inequality, we obtain
 \begin{equation}\notag
  \E|\varepsilon_{j3}|^q\le C^qq^qn^{-q}(\mu_4^{\frac q2}\E(\sum_{l\in\mathbb T_j}|R^{(j)}_{ll}|^2)^{\frac q2}+\mu_{2q}\sum_{l\in\mathbb T_j}\E|R^{(j)}_{ll}|^q).
 \end{equation}
Using that $|R^{(j)}_{ll}|\le \frac14$ and $\mu_{2q}\le D^{2q-4}n^{\frac q2-1}\mu_4$, we get
\begin{equation}\notag
 \E|\varepsilon_{j3}|^q\le C^qq^qn^{-\frac q2}.
\end{equation}
Thus Lemma \ref{eps3} is proved.
 
\end{proof}
\begin{lem}\label{etaj}Assuming the conditions of Theorem \ref{main}, we get
\begin{align}\notag
 \E|\eta_{j}|^q\le \frac {C^qq}{n^{\frac q2}}.
\end{align}

\end{lem}
\begin{proof}
 The proof is similar to proof of Lemma \ref{eps2}. We need to use that $|[(R^{(j)})^2]_{kl}|\le V^{-2}=\frac1{16}$ and 
 $\sum_{l\in\mathbb T_j}|[(R^{(j)})^2]_{kl}|^2\le V^{-4}$.
\end{proof}

\begin{lem}\label{lem2} Assuming the conditions of Theorem \ref{main}, we get, for any $q\ge1$,
\begin{align}\notag
 |\varepsilon_{j4}|^q\le \frac {C^q}{n^{ q}}.
\end{align}

\end{lem}
\begin{proof}
 The result follows immediately from the bound
 \begin{equation}\notag
  |\varepsilon_{j4}|\le \frac1{nv}, \text{ a. s.}
 \end{equation}
See for instance \cite{GT:2003},  Lemma 3.3.
\end{proof}

Now we investigate the behavior of $R_{jj}-s(z)$ for $z=u+iV$ with $V=4$.
\begin{lem}\label{bigv}
 Assuming the conditions of Theorem \ref{main}, we get, 
 \begin{equation}\notag
  \E|R_{jj}-s(z)|^4\le Cn^{-2}.
 \end{equation}

\end{lem}
\begin{proof}
 By equality \eqref{repr001} we have
 \begin{equation}\notag
  \E|R_{jj}-s(z)|^4\le C(\E|\Lambda_n|^4+\sum_{\nu=1}^4\E|\varepsilon_{j\nu}|^4).
 \end{equation}
By equation \eqref{main3}, for $V=4$,
\begin{equation}\notag
 \E|\Lambda_n|^4\le C\E|T_n|^4\le \frac Cn\sum_{l=1}^n\E|\varepsilon_l|^4.
\end{equation}

Direct calculations show that
\begin{equation}\notag
 \E|\varepsilon_{j2}|^4\le C\mu_4^2n^{-2}\E(\frac1n\sum_{l,k=1}^n|R_{lk}^{(j)}|^2)^2\le \frac C{n^2}.
\end{equation}
Similarly we get
\begin{equation}\notag
 \E|\varepsilon_{j3}|^4\le C\mu_4n^{-2}\E(\frac1n\sum_{l\in\mathbb T_j}|R_{ll}^{(j)})^2+C\mu_4n^{-2}\frac1n\sum_{l\in\mathbb T_j}\E|R^{(j)}_{ll}|^4\le Cn^{-2}.
\end{equation}
Finally, by Lemma \ref{lem2}, we have
\begin{equation}\notag
 \E|\varepsilon_{j4}|^4\le Cn^{-4}.
\end{equation}
Combining these inequalities we get the claim.
Thus Lemma \ref{bigv} is proved.
\end{proof}

\subsection{Some Auxiliary Bounds for Resolvent Matrices for $z\in\mathbb G$}Introduce now the region
\begin{align}\label{region}
 \mathbb G &:=\{z=u+iv\in\mathbb C^+:\,u\in\mathbb J_{\varepsilon}, v\ge v_0/\sqrt{\gamma}\},\;\; \text{where} \; v_0= A_0 n^{-1},\\
  \mathbb J_{\varepsilon} &=[-2+\varepsilon,2-\varepsilon], \quad \varepsilon:=c_1n^{-\frac23}, \quad \gamma=\gamma(u)=
  \min\{2-u, 2+u\}.\notag
\end{align}
In the next lemma we some simple inequalities for the region $\mathbb G$

\begin{lem}\label{lemG}
 For any $z\in\mathbb G$ we have
 \begin{align}
  |z^2-4|\ge 2\max\{\gamma,v\},\qquad
  {nv}\sqrt{|z^2-4|}\ge 2A_0.
 \end{align}

\end{lem}
\begin{proof}
 We observe that
 \begin{equation}
  |z^2-4|=|z-2||z+2|\ge 2\sqrt{\gamma^2+v^2}.
 \end{equation}
This inequality proves the Lemma.
\end{proof}

\begin{lem}\label{lem00}Assuming the conditions  of Theorem \ref{main}, there exists an absolute constant $c_0>0$ such that for any $\mathbb J\subset \mathbb T$,
\begin{equation}\label{lem00.1}
 |z+m_n^{(\mathbb J)}(z)+s(z)|\ge \im m_n^{(\mathbb J)}(z),
\end{equation}
moreover, for $z\in\mathbb G$,
\begin{equation}\label{lem00.2}
 |z+m_n^{(\mathbb J)}(z)+s(z)|\ge c_0\sqrt{|z^2-4|}.
\end{equation}

\end{lem}

\begin{proof}
 First we note
 \begin{equation}
 |z+m_n^{(\mathbb J)}(z)+s(z)|\ge\im(z+s(z))\ge \frac12\im\sqrt{z^2-4}.
 \end{equation}
Furthermore, it is simple to check that, for $z=u+iv$ with $v>0$
\begin{equation}
 \im\sqrt{z^2-4}\ge \frac{\sqrt2}2\sqrt{|z^2-4|}.
\end{equation}
Thus Lemma \ref{lem00} is proved.
\end{proof}



\begin{lem}\label{basic1}Assuming the conditions of Theorem \ref{main}, there exists an absolute constant $C>0$ such that for any $j=1,\ldots,n$,
\begin{equation}
 \E\{|\varepsilon_{j1}|^4\big|\mathfrak M^{(j)}\}\le \frac {C\mu_4}{n^2}.
\end{equation}

\end{lem}
\begin{proof}
 The result follows immediately from the definition of $\varepsilon_{j1}$.
\end{proof}
\begin{lem}\label{basic2}Assuming the conditions of Theorem \ref{main}, there exists an absolute constant $C>0$ such that for any $j=1,\ldots,n$, 
 \begin{equation}\label{basic3}
  \E\{|\varepsilon_{j2}|^2\big|\mathfrak M^{(j)}\}\le \frac C{nv}\im m_n^{(j)}(z),
 \end{equation}
and
\begin{equation}\label{basic4}
 \E\{|\varepsilon_{j2}|^4\big|\mathfrak M^{(j)}\}\le \frac {C\mu_4^2}{n^2v^2}\im^2m_n^{(j)}(z).
\end{equation}

\end{lem}
\begin{proof}
 Note that r.v.'s $X_{jl}$, for $l\in\mathbb T_j$ are independent of $\mathfrak M^{(j)}$ and that for $l,k\in\mathbb T_j$ $R^{(j)}_{lk}$
 are measurable with respect to $\mathfrak M^{(j)}$.
 This implies that $\varepsilon_{j2}$ is a quadratic form with coefficients $R^{(j)}_{lk}$ independent of $X_{jl}$. Thus
 its variance and fourth moment are easily available.
 \begin{equation}
  \E\{|\varepsilon_{j2}|^2\big|\mathfrak M^{(j)}\}=\frac1{n^2}\sum_{l\ne k\in\mathbb T_j}|R^{(j)}_{lk}|^2\le \frac1{n^2}\Tr\mathbb |\mathbf R^{(j)}|^2,
 \end{equation}
Here we use the notation $|\mathbf A|^2=\mathbf A\mathbf A^*$ for any matrix $\mathbf A$.
Applying Lemma \ref{resol00}, inequality \eqref{res1},
we get equality \eqref{basic3}.

Furthermore, direct calculations show that
\begin{align}
 \E\{|\varepsilon_{j2}|^4\big|\mathfrak M^{(j)}\}&\le \frac C{n^2}(\frac1n\sum_{l\ne k\in\mathbb T_j}|R^{(j)}_{lk}|^2)^2
 +\frac {C\mu_4^2}{n^2}\frac1{n^2}\sum_{l\in\mathbb T_{j}}|R^{(j)}_{lk}|^4\notag\\&\le \frac {C\mu_4^2}{n^2}(\frac1n\sum_{l\ne k\in\mathbb T_j}|R^{(j)}_{lk}|^2)^2
 \le\frac{C\mu_4^2}{n^2v^2}(\im m_n^{(j)}(z))^2.
\end{align}
Here again we used Lemma \ref{resol00}, inequality \eqref{res1}.
Thus Lemma \ref{basic2} is proved.
\end{proof}
\begin{lem}\label{basic5}Assuming the conditions of Theorem \ref{main}, there exists an absolute constant $C>0$ such that for any $j=1,\ldots,n$,
 \begin{align}\label{basic6}
  \E\{|\varepsilon_{j3}|^2\big|\mathfrak M^{(j)}\}\le \frac {C\mu_4}{n}\frac1n\sum_{l\in\mathbb T_j}|R^{(j)}_{ll}|^2,
 \end{align}
 and
 \begin{align}\label{basic7}
  \E\{|\varepsilon_{j3}|^4\big|\mathfrak M^{(j)}\}\le \frac {C}{n^2}(\frac1n\sum_{l\in\mathbb T_j}|R^{(j)}_{ll}|^2)^2
  +\frac {C\mu_4}{n^2}\frac1n\sum_{l\in\mathbb T_j}|R^{(j)}_{ll}|^4.
 \end{align}

\end{lem}
\begin{proof}The first inequality is obvious. To prove the second inequality, we apply
 Rosenthal's inequality. We obtain
\begin{align}
\E\{|\varepsilon_{j3}|^4\big|\mathfrak M^{(j)}\}\le \frac {C\mu_4}{n^2}(\frac1n\sum_{l\in\mathbb T_j}|R^{(j)}_{ll}|^2)^2
+\frac {C\mu_8}{n^3}\frac1n\sum_{l\in\mathbb T_j}|R^{(j)}_{ll}|^4. 
\end{align}
Using  $|X_{jl}|\le Cn^{\frac 14}$ we get 
$\mu_8\le Cn\mu_4$ and the claim.
Thus Lemma \ref{basic5} is proved.
\end{proof}
\begin{cor}\label{corgot}Assuming the conditions of Theorem \ref{main}, there exists an absolute constant $C>0$, depending on $\mu_4$ and $D$ only, 
such that for any $j=1,\ldots,n$,  $\nu=1,2,3$
$z\in\mathbb G$, and $1\le \alpha\le \frac12A_1(nv)^{\frac14}$
\begin{align}
 \E\frac{|\varepsilon_{j\nu}|^2}{|z+m_n^{(j)}(z)+s(z)||z+m_n^{(j)}(z)|^{\alpha}}\le \frac C{nv}
\end{align}
and 
 \begin{align}\label{corgot1}
 \E\frac{|\varepsilon_{j\nu}|^4}{|z+m_n^{(j)}(z)+s(z)|^2|z+m_n^{(j)}(z)|^{\alpha}}\le \frac C{n^2v^2}.
\end{align}
\end{cor}
\begin{proof}
 For $\nu=1$, by Lemma \ref{lem00}, we have
 \begin{equation}
  \E\frac{|\varepsilon_{j\nu}|^2}{|z+m_n^{(j)}(z)+s(z)||z+m_n^{(j)}(z)|^{\alpha}}\le \frac1{n\sqrt{|z^2-4|}}\E|X_{jj}|^4\E\frac1{|z+m_n^{(j)}(z)|^{\alpha}}.
 \end{equation}
Applying now Corollary \ref{cor8}, we get the claim. The proof of the second inequality for $\nu=1$ is similar.
For $\nu=2$ we apply Lemma \ref{basic2}, inequality \eqref{basic3} and obtain, using that $\im m_n^{(j)}(z)\le |z+m_n^{(j)}(z)+s(z)|$, (see \eqref{lem00.1}),
\begin{align}
  \E\frac{|\varepsilon_{j2}|^2}{|z+m_n^{(j)}(z)+s(z)||z+m_n^{(j)}(z)|^{\alpha}}&\le 
  \frac1{nv}\E\frac{\im m_n^{(j)}(z)}{|z+m_n^{(j)}(z)+s(z)||z+m_n^{(j)}(z)|^{\alpha}}\notag\\&\le \frac C{nv}\E\frac{1}{|z+m_n^{(j)}(z)|^{\alpha}}.
 \end{align}
Similarly, using Lemma \ref{basic2}, inequality \eqref{basic4}, we get
 \begin{align}
  \E\frac{|\varepsilon_{j2}|^4}{|z+m_n^{(j)}(z)+s(z)|^2|z+m_n^{(j)}(z)|^{\alpha}}&\le 
  \frac C{n^2v^2}\E\frac{\im^2 m_n^{(j)}(z)}{|z+m_n^{(j)}(z)+s(z)|^2|z+m_n^{(j)}(z)|^{\alpha}}\notag\\&\le \frac C{n^2v^2}\E\frac{1}{|z+m_n^{(j)}(z)|^{\alpha}}.
 \end{align}
 Applying Corollary \ref{cor8}, we get the claim.
 For $\nu=3$, we apply Lemma \ref{basic5}, inequalities \eqref{basic6} and \eqref{basic7} and Lemma \ref{lem00}. We get
 \begin{align}
  \E&\frac{|\varepsilon_{j3}|^2}{|z+m_n^{(j)}(z)+s(z)||z+m_n^{(j)}(z)|^{\alpha}}\notag\\&\qquad\qquad\qquad\qquad\qquad\qquad\le 
  \frac C{n\sqrt{|z^2-4|}}\E\frac{1}{|z+m_n^{(j)}(z)|^{\alpha}}\Big(\frac1n\sum_{l\in\mathbb T_j}|R^{(j)}_{ll}|^2\Big),
 \end{align}
and
\begin{align}
  \E&\frac{|\varepsilon_{j3}|^4}{|z+m_n^{(j)}(z)+s(z)|^2|z+m_n^{(j)}(z)|^{\alpha}}\notag\\&\qquad\qquad\qquad\qquad\qquad\qquad\le 
  \frac C{n^2{|z^2-4|}}\E\frac{1}{|z+m_n^{(j)}(z)|^{\alpha}}\Big(\frac1n\sum_{l\in\mathbb T_j}|R^{(j)}_{ll}|^4\Big).
 \end{align}
 Using now the Cauchy -- Schwartz inequality and Corollary \ref{cor8}, we get the claim.
\end{proof}


\begin{lem}\label{basic8}Assuming the conditions of Theorem \ref{main}, there exists an absolute constant $C>0$ such that for any $j=1,\ldots,n$,
 \begin{align}
  |\varepsilon_{j4}|\le \frac C{nv}\quad\text{a.s.}
 \end{align}
 \end{lem}
\begin{proof}This inequality follows from 
\begin{equation}\label{shur}
 \Tr \mathbf R-\Tr\mathbf R^{(j)}=(1+\frac1n\sum_{l,k\in\mathbb T_j}X_{jl}X_{jk}[(R^{(j)})^2]_{kl})R_{jj}=R^{-1}_{jj}\frac {dR_{jj}}{dz},
\end{equation}
which may be obtained using the Schur complement formula.
For details see, for instance \cite{GT:2003}, Lemma  3.3.
 
\end{proof}
\begin{lem}\label{lam1}Assuming the conditions of Theorem \ref{main}, we have, for  $z\in\mathbb G$,
\begin{align}
\E|\Lambda_n|^2\le \frac C{nv|z^2-4|^{\frac12}}.
\end{align}
\end{lem}
\begin{proof}We write
\begin{align}
 \E|\Lambda_n|^2=\E\Lambda_n\overline \Lambda_n=\E\frac{T_n}{z+m_n(z)+s(z)}\overline\Lambda_n=\sum_{\nu=1}^4\E\frac{T_{n\nu}}{z+m_n(z)+s(z)}\overline\Lambda_n,
\end{align}
where
\begin{align}
 T_{n\nu}:=\frac1n\sum_{j=1}^n\varepsilon_{j\nu}R_{jj}, \text{ for }\nu=1,\ldots,4.
\end{align}
Applying the Cauchy -- Schwartz inequality , we get
\begin{equation}\label{ll2}
 \E^{\frac12}|\Lambda_n|^2\le \sum_{\nu=1}^4\E^{\frac12}\frac{|T_{n\nu}|^2}{|z+m_n(z)+s(z)|^2}.
\end{equation}

First we observe that by \eqref{shur}
\begin{align}
 |T_{n4}|=\frac1n|m_n'(z)|\le \frac1{nv}\im m_n(z).
\end{align}
 Hence $|z+m_n(z)+s(z)|\ge \im m_n(z)$ and Jensen's inequality yields
\begin{equation}\label{tn4}
 \E\frac{|T_{n4}|^2}{|z+m_n(z)+s(z)|^2}\le \frac 1{n^2v^2}.
\end{equation}
Furthermore, we observe that, 
\begin{align}
\frac1{|z+s(z)+m_n(z)|}&\le\frac1{|z+s(z)+m_n^{(j)}(z)|}(1+
\frac{|\varepsilon_{j4}|}{|z+s(z)+m_n(z)|} ).
\end{align}
Therefore, by Lemmas \ref{basic8} and  \ref{lemG}, for $z\in\mathbb G$,
\begin{align}\label{lar1}
\frac1{|z+s(z)+m_n(z)|}&\le\frac C{|z+s(z)+m_n^{(j)}(z)|}.
\end{align}
Applying inequality \eqref{lar1}, we may write
\begin{align}
 \E\frac{|T_{n\nu}|^2}{|z+m_n(z)+s(z)|^2}\le \frac1n\sum_{j=1}^n\E\frac{|\varepsilon_{j\nu}|^2|R_{jj}|^2}{|z+s(z)+m_n^{(j)}(z)|^2}.
\end{align}
Applying Cauchy -- Schwartz inequality and Lemma \ref{lem00}, we get
\begin{align}
 \E\frac{|T_{n\nu}|^2}{|z+m_n(z)+s(z)|^2}\le \frac C{n|z^2-4|^{\frac12}}\sum_{j=1}^n\E^{\frac12}\frac{|\varepsilon_{j\nu}|^4}{|z+m_n^{(j)}(z)+s(z)|^2}
 \E^{\frac12}|R_{jj}|^4.
\end{align}
Using now Corollary \ref{corgot}, inequality \eqref{corgot1} and Corollary \ref{cor8}, we get for $\nu=1,2,3$
\begin{align}\label{lll2}
 \E\frac{|T_{n\nu}|^2}{|z+m_n(z)+s(z)|^2}\le \frac C{nv|z^2-4|^{\frac12}}.
\end{align}
Inequalities \eqref{ll2}, \eqref{tn4} and \eqref{lll2} together complete the proof.
Thus Lemma \ref{lam1} is proved.
\end{proof}



\end{document}